\documentclass[fleqn,12pt]{article}
\usepackage{amsmath,amssymb,amsthm,framed,esint,xcolor,dsfont}
\usepackage[english]{babel}
\usepackage[margin=3cm]{geometry}
\usepackage{csquotes}
\usepackage[bookmarks]{hyperref}
\usepackage{enumitem}  

\usepackage[title]{appendix}

\newcommand{\norm}[1]{{\Vert #1\Vert}}
\newcommand{\abs}[1]{{\left\vert #1\right\vert}}
\newcommand{\R}{{\mathbb R}}
\newcommand{\C}{{\mathbb C}}
\newcommand{\Z}{{\mathbb Z}}

\newcommand{\ep}{\varepsilon}
\newcommand{\lt}{\left}
\newcommand{\rt}{\right}
\newcommand{\na}{\nabla}
\newcommand{\ti}{\tilde}
\newcommand{\nn}{\nonumber}
\newcommand{\e}{\varepsilon}
\newcommand{\qd}{\quad}
\newcommand{\wt}{\widetilde}

\DeclareMathOperator{\supp}{supp}
\DeclareMathOperator{\dist}{dist}
\newcommand{\loc}{{\rm loc}}

\UndeclareTextCommand{\textcommabelow}{T1}
\DeclareTextCommandDefault{\textcommabelow}[1]{\fontencoding{QX}\selectfont\textcommabelow
#1}

\newtheorem{thm}{Theorem}[section]
\newtheorem{prop}[thm]{Proposition}
\newtheorem{lem}[thm]{Lemma}

\theoremstyle{definition}
\newtheorem{rem}[thm]{Remark}
\newtheorem*{rem*}{Remark}
\newtheorem{ex}[thm]{Example}

\numberwithin{equation}{section}

\title{On  regularity and  rigidity of $2\times 2$ differential inclusions into non-elliptic curves}
\date{}
\author{Xavier Lamy\footnote{Institut de Math\'ematiques de Toulouse, UMR 5219, Universit\'e de Toulouse, CNRS, UPS
IMT, F-31062 Toulouse Cedex 9, France. Email: Xavier.Lamy@math.univ-toulouse.fr}
\and Andrew Lorent\footnote{Department of Mathematical Sciences, University of Cincinnati, Cincinnati, OH 45221, USA. Email: lorentaw@uc.edu} 
\and Guanying Peng\footnote{Department of Mathematical Sciences, Worcester Polytechnic Institute, Worcester, MA 01609, USA. Email: gpeng@wpi.edu}}
\begin{document}

\maketitle

\begin{abstract}
We study differential inclusions $Du\in \Pi$ in an open set $\Omega\subset\mathbb R^2$, where $\Pi\subset \mathbb R^{2\times 2}$ is a 
compact connected
 $C^2$ curve without rank-one connections, 
but non-elliptic:
 tangent lines to $\Pi$ may have rank-one connections,
 so that
 classical regularity and rigidity results do not apply.
For a 
wide
 class of such curves $\Pi$, 
we show that $Du$ is locally Lipschitz outside a discrete set, and is rigidly characterized around each singularity.
Moreover, 
in the partially elliptic case where at least one tangent line to $\Pi$ has no rank-one connections,
or under some topological restrictions on the tangent bundle of $\Pi$,
there are no singularities.
This goes well beyond
 previously known particular cases 
 related to Burgers' equation and to the Aviles-Giga functional.
The key is the identification and 
 appropriate use of a general underlying structure:
 an infinite family of conservation laws,
 called entropy productions in reference to the theory of scalar conservation laws.
\end{abstract}

\section{Introduction}\label{s:intro}

Let $\Omega\subset\R^2$ be an open set.
We 
demonstrate
 regularity and rigidity properties of weakly differentiable maps $u\colon \Omega\to\R^2$ satisfying the differential inclusion
\begin{align*}
Du \in \Pi\qquad\text{a.e. in }\Omega,
\end{align*}
where $\Pi\subset\R^{2\times 2}$ is a 
compact connected $C^2$ curve without rank-one connections, which is 
non-elliptic:
tangent lines to $\Pi$ may be generated by rank-one matrices. (Here and in the rest of the article, by curve we mean a one-dimensional submanifold, with or without boundary, in other words it is always embedded.) 

Regularity of differential inclusions is a subject with a long history.  The best known result is the analyticity of solutions of the Cauchy-Riemann equations -- reformulated as a differential inclusion, this is the statement that if a function 
$u:\Omega\rightarrow \mathbb{R}^2$ satisfies
$Du\in\mathrm{CO}_+(2)=\mathbb R_+ \mathrm{SO}(2)$
 everywhere, then $u$ is analytic. In 1850 Liouville \cite{LIO50} proved that if a $C^3$ function $u$ satisfies the differential inclusion
  $Du\in \mathrm{CO}_{+}(3)=\mathbb R_+ \mathrm{SO}(3)$,
 then it is 
a M\"{o}bius mapping.
The optimal generalization of this result is still an open problem, 
 which 
  has inspired a large literature and to some extent motivated the development of the theory of higher dimensional quasiconformal mappings  \cite{boj,reshetnyak94}.

It is well known (see e.g. \cite{muller-lecture})
 that a necessary condition for  regularity of a differential inclusion $Du\in K$
 is that $K$ should have no rank-one connections: 
 \begin{align*}
 \mathrm{rank}(A-B)\neq 1\quad\text{for all matrices }A\ne B\in K.
 \end{align*}
This condition is by no means sufficient, a differential inclusion without rank-one connection may have wild solutions.
The existence and structure of such wild solutions and their weak limits has been the object of an extremely rich line of research,
with many far-reaching applications in the theory of nonlinear PDEs.
We do not attempt to describe these results here, and refer instead  to the articles \cite{KMS03,DLS22} and the references therein.

Here we focus on the essentially orthogonal question of finding general sufficient conditions on $K$ which ensure 
higher regularity of solutions of the differential inclusion $Du\in K$.
For  differential inclusions in $\R^{2\times 2}$, a
first general sufficient condition is due to
\v{S}ver\'{a}k:
if $K\subset \mathbb{R}^{2\times 2}$ is a smooth connected closed submanifold without rank-one connections that is \em elliptic \rm 
(in the sense that its tangent spaces have no rank-one connections), then solutions of the differential inclusion $Du\in K$ are smooth \cite[\S~5]{sverak93}.

For non-elliptic sets, we are not aware of any general regularity result, but we
describe next two examples.

\begin{ex}\label{ex:burgers}
Consider
 a bounded weak solution of Burgers' equation, which does not dissipate energy:
\begin{align*}
\partial_t v +\partial_x \frac{v^2}{2}=0,\quad 
\partial_t \frac{v^2}{2}
+\partial_x \frac{v^3}{3}=0.
\end{align*}
According to \cite{panov94} (see also \cite{DLOW03minentcond}), the function $v$ is both an entropy and anti-entropy solution of Burgers' equation, and by Oleinik's one-sided Lipschitz estimate must therefore be locally Lipschitz. 
In a simply connected domain, 
the two conservation laws satisfied by $v$ are equivalent to the existence of $u_1,u_2$ such that 
\begin{align*}
Du
=
\left(
\begin{array}{cc}
\partial_t u_1 & \partial_x u_1
\\
\partial_t u_2 & \partial_x u_2
\end{array}
\right)
=
\left(
\begin{array}{cc}
-\frac{v^2}{2} & v
\\
-\frac{v^3}{3} & \frac{v^2}{2}
\end{array}
\right):=\gamma(v),
\end{align*}
so the Lipschitz regularity of $v$ amounts to Lipschitz regularity of $Du$ if $Du\in\Pi =\gamma([a,b])$. 
It can be checked that $\det(\gamma')=0$, so this differential inclusion is nowhere elliptic.
It can however also be checked that it is not too degenerate, in the sense that it satisfies the assumption \eqref{eq:det4} of Theorem \ref{t:reg_gen} below.
\end{ex}

\begin{ex}\label{ex:AG}
In our previous work \cite{LLP20}, motivated by connections with the Aviles-Giga functional \cite{AG99,ADLM99,JK00,JOP02,LP18},  we studied an explicit closed curve
 $K_0\subset\R^{2\times 2}$ 
  which has no rank-one connections, but is nowhere elliptic.
(The precise definition of $K_0$ is recalled in \S~\ref{ss:ex_intro}.)
There we established that solutions of $Du\in K_0$ enjoy some regularity: $Du$ is locally Lipschitz outside a discrete set -- but also some rigidity: $Du$ is explicitly characterized in any convex neighborhood of a singularity. 
Again, it can be checked that $K_0$ is not too degenerate
 in the sense that it satisfies the assumption \eqref{eq:det4} of Theorem \ref{t:reg_gen} below, a fact crucially used in \cite{LP18,LLP20}.
\end{ex}

Our proof in \cite{LLP20} 
relied strongly on the explicit form of $K_0$ and its link with the eikonal equation,
but since then we have been intrigued by the possibility that there might be a general result for differential inclusions into curves that do not have rank-one connections and are not necessarily elliptic.  

This is what we establish in this article:  we prove regularity and rigidity for differential inclusions $Du\in \Pi$, where $\Pi\subset\R^{2\times 2}$ is a generic 
compact connected
 $C^2$ curve which has no rank-one connections but may not be elliptic.
More precisely, a 
compact connected
 curve $\Pi\subset\R^{2\times 2}$ without rank-one connections is elliptic if and only if the quadratic estimate
\begin{align*}
|\det(A-B)|\geq c\: |A-B|^2\qquad\forall A, B\in \Pi,
\end{align*}
is valid for some $c>0$ \cite[\S~5]{sverak93}.
Here we assume only a weaker quartic estimate \eqref{eq:det4}, which allows the tangent lines to have rank-one connections, 
while retaining some weak nondegeneracy. 
Our main theorem is

\begin{thm}\label{t:reg_gen}
Let $\Pi  \subset\R^{2\times 2}$  be a
compact connected 
$C^2$ curve, with or without boundary.
Assume that $\Pi$ has no
 rank-one connections,
and that it satisfies
 the nondegeneracy estimate
\begin{align}\label{eq:det4}
|\det(A-B)|\geq c\, |A-B|^4\qquad\forall A,B\in\Pi,
\end{align}
for some constant $c>0$. For any open set $\Omega\subset\R^2$ and weakly differentiable map $u\colon\Omega\to\R^2$, if $u$ solves the differential inclusion
\begin{align}\label{eq:diffinc}
Du\in\Pi\qquad\text{a.e. in }\Omega,
\end{align}
then $Du$ is locally Lipschitz away from a locally finite set $\mathcal S$. 
Moreover, the singular set $\mathcal S$ is empty 
\color{black}
in the following (non-disjoint) cases:
\color{black}
\begin{itemize}
\item if $\Pi$
 is partially elliptic (at least one tangent line to $\Pi$ has no rank-one connections);
 \item 
if $\Pi$ has a boundary; 
 \item under some topological conditions on the tangent bundle $T\Pi$, to be made explicit in Theorem~\ref{t:NE}.
 \end{itemize}
\end{thm}

It is  worth recalling here the powerful observation of Zhang in \cite{zhang97} that any connected set $K\subset\R^{2\times 2}$ without rank-one connection is contained in the graph of a Lipschitz mapping between conformal and anticonformal matrices.
This fact is at the basis of important advances on properties of $2\times 2$ differential inclusions, see e.g. \cite{FS08},
and the perspective it gives is implicit in many of our arguments.

\begin{rem}\label{r:reg_Pi}
Theorem~\ref{t:reg_gen} requires the curve $\Pi$ to be of class $C^2$.
If we assume that $\Pi$ is smoother, say $C^3$, 
some small parts of our arguments can be simplified.
For instance, the proof of Lemma~\ref{l:detcontrols4} is shorter (as explained in \textsection~\ref{s:det4NE}), 
and Lemma~\ref{l:conslawsC1} is superfluous.
Moreover, it is natural to wonder whether  the $C^2$ regularity assumption on $\Pi$ could be relaxed -- to merely $C^1$ or some intermediate regularity.
We do use the $C^2$ assumption in a quite crucial way, for instance in Lemma~\ref{l:detcontrols4} and for the commutator estimates in \textsection~\ref{s:conslaws}, 
and dealing with a less regular $\Pi$ would certainly have to involve refined arguments.
\end{rem}

\begin{rem}\label{r:nondeg_assumption}
Without the nondegeneracy assumption \eqref{eq:det4}, $Du$ might fail to be Lipschitz away from singularities. 
Consider for instance, for $q > 0$, the scalar conservation law
\begin{align*}
\partial_t v +\partial_x \frac{v^2|v|^q }{2+q} = 0,
\end{align*}
whose characteristic curves are given by
$
x=x_0 +t v_0 |v_0|^q.
$
For any bounded continuous
and 
 monotone nondecreasing initial condition $v_0(x)$, there is a solution $v(t,x)$ constant along characteristics for $t>0$.
That solution also solves
\begin{align*}
\partial_t \frac{v^2}{2} +\partial_x \frac{v|v|^{q+2}}{3+q} =0,
\end{align*}
and provides a solution of the differential inclusion 
\begin{align*}
Du\in 
 \Pi := \left\lbrace
\left(
\begin{array}{cc}
-\frac{w^2|w|^q}{2+q} & w
\\
-\frac{w |w|^{q+2}}{3+q} & \frac{w^2}{2}
\end{array}
\right)
\colon |w|\leq \|v_0\|_\infty
\right\rbrace,
\end{align*}
which satisfies
\begin{align}\label{eq:det4q}
\det(A-B)\geq c |A-B|^{4+q}\qquad\forall A,B\in\Pi.
\end{align}
 Arguing as in \cite[Proposition~4.3]{COW08}, 
 one can
  choose the initial data $v_0$ in a way that 
  $v(t,x)$ is Hölder continuous with Hölder exponent not better than $1/(1+q)$.
Zero-energy states of generalized Aviles-Giga energies \cite{BP17,LLP22}, 
would also provide closed curves $\Pi$ satisfying \eqref{eq:det4q} and solutions of $Du\in\Pi$ for which $Du$ is not better than $1/(1+q)$-Hölder continuous, 
see in particular \cite[Remark~4.3]{BP17}.
\end{rem}

In view of Remark~\ref{r:nondeg_assumption}, 
the following open question is very natural:

\medskip
\noindent\textbf{Question:} If $\Pi\subset\R^{2\times 2}$ is a compact connected $C^2$ curve satisfying \eqref{eq:det4q} for some $q>0$, can one deduce that any solution $u$ of $Du\in\Pi$ a.e. in an open set $\Omega\subset\R^2$ is locally $C^{1,\frac{1}{1+q}}$  away from a locally finite set?

In fact, one could even ask:  if $\Pi\subset\R^{2\times 2}$ is a compact connected $C^2$ curve without rank-one connection, can one deduce that any solution of $Du\in\Pi$ is locally $C^1$ away from a locally finite set?

\medskip

Our method fails to answer these questions because 
we bootstrap from an initial low regularity $Du\in B^{1/3}_{p,\infty,\loc}$ for some $p>3$, which we do not know how to obtain if \eqref{eq:det4} is not satisfied (see Lemma~\ref{l:besov}).

\color{black}
The heart of Theorem~\ref{t:reg_gen} is the nowhere elliptic case, 
where all tangent lines to $\Pi$ have rank-one connections.
The partially elliptic case can essentially be reduced to the case of a nowhere elliptic curve with boundary : this is one of the reasons why it is important to treat curves with a boundary.

In the rest of this introduction, 
we present the precise statements concerning the reduction
to nowhere elliptic curves and
 the structure of solutions in that case.
Applying a homothety, we assume without loss of generality that $\Pi$ has length 
at most $2\pi$.
We fix $\gamma\colon  I \to\Pi$ an arc-length parametrization of class $C^2$,
where $I$ is either a segment $[a,b]\subset\R$ ($a<b<a+2\pi$) in the case with boundary, or $I=\R/2\pi\Z$ in the case without boundary. 
\color{black}

\subsection{\color{black}Reduction to nowhere elliptic curves}\label{ss:PE}

We  divide the curve $\Pi=\gamma(I)$  into elliptic and non-elliptic points,
\begin{align*}
\Pi
&=\Pi_E \cup \Pi_{NE},\qquad
\Pi_E
=\left\lbrace \gamma(t)\colon \det(\gamma'(t))\neq 0\right\rbrace,\qquad \Pi_{NE}=\Pi\setminus \Pi_E\,.
\end{align*}

\begin{thm}\label{t:PE}
Let $\Pi\subset\R^{2\times 2}$ be a 
compact connected  $C^2$ curve without rank-one connections.

\begin{itemize}
\item[(i)]
For any connected open set $\Omega\subset\R^2$ and weakly differentiable map $u\colon\Omega\to\R^2$, if $u$ solves the differential inclusion
\begin{align*}
Du\in\Pi
\qquad\text{a.e. in }\Omega,
\end{align*}
then 
 either $Du\in \Pi_{NE}$ a.e. or $Du$ is constant.

\item[(ii)]
Assume moreover that $\Pi$ satisfies the quartic nondegeneracy estimate \eqref{eq:det4}, then
\begin{itemize}
\item[(a)] either $Du$ is constant,

\item[(b)] 
or  
 $Du$  
takes values into one single connected component of $\Pi_{NE}$.
\end{itemize}
\end{itemize}
\end{thm}

In the partially elliptic case $\Pi_E\neq \emptyset$, 
each connected component of $\Pi_{NE}$ is a nowhere elliptic curve with non-empty boundary.
Therefore, 
 the regularity of $Du$ in case $(ii.b)$ of Theorem~\ref{t:PE} will be given by Theorem~\ref{t:NE_bdry} below.

Part $(i)$ of Theorem~\ref{t:PE} follows
 from the unique continuation principle established recently in \cite{DPGT23}.
(A weaker version of this result, namely $Du$ is locally constant in $\Omega_E=(Du)^{-1}(\Pi_E)$, which is sufficient for our purpose, could also be deduced
from regularity properties of degenerate elliptic equations in two variables 
 established in \cite[Theorem~1.9]{lledos23}.)
Under 
a low regularity assumption which is implied by
 the quartic nondegeneracy estimate \eqref{eq:det4}, this first conclusion can be strengthened to $Du$ being either constant or with values into a single connected component of $\Pi_{NE}$.
%

\subsection{Nowhere elliptic curves without boundary}\label{ss:NE_closed_intro}

{\color{black}
Thanks to Theorem~\ref{t:PE}, the proof of Theorem~\ref{t:reg_gen} reduces to the nowhere elliptic case, 
which we now proceed to describe more precisely,
starting with curves without boundary.
Consider a closed $C^2$ curve 
$\Pi=\gamma(\R/2\pi\Z)$ 
 that is nowhere elliptic.
 In terms 
 of the arc-length parametrization $\gamma$, this} amounts to
\begin{align*}
\det( \gamma'(\theta))=0\qquad\forall \theta\in\R.
\end{align*}
This is the only case where the differential inclusion \eqref{eq:diffinc} may develop singularities.
Specific details about rigidity of singularities depend on 
 topological properties of 
the tangent bundle $T\Pi$.

\color{black}
More specifically, this tangent bundle 
 induces a loop  into
 the projective line $\mathbb{RP}^1=\mathbb S^1/\lbrace\pm 1\rbrace$,
 as follows.
The matrix $(\mathrm{cof}\,\gamma'(\theta))^T$ has rank-one, 
hence its range is spanned by a unique 
 $\Psi(e^{i\theta})\in\mathbb{RP}^1$. 
The map $\Psi\colon\mathbb S^1\to\mathbb {RP}^1$
is the desired loop, and it can be explicitly expressed in terms of
$\gamma$, see Lemma~\ref{l:lambdaPsi} and Remark~\ref{r:Psi}.
This loop carries a winding number $\deg(\Psi)$, which characterizes its homotopy class in $\pi_1(\mathbb{RP}^1)$.
\color{black}

We adopt here, following e.g. \cite[\S~VIII.B]{BCL86}, the convention that the winding number is a half-integer: 
it is given by 
\begin{align}\label{eq:degPsi}
\deg(\Psi)=\frac{\varphi_\Psi(2\pi)-\varphi_\Psi(0)}{2\pi} \in \frac 12\Z,
\end{align}
for any continuous phase $\varphi_\Psi\colon\R\to\R$ such that
$\Psi(e^{i\theta})=\lbrace\pm e^{i\varphi_\Psi(\theta)}\rbrace$.
The map $\Psi$ is orientable, that is, can be lifted to a $C^1$ map $\Psi\colon\mathbb S^1\to\mathbb S^1$, if and only if $\deg(\Psi)\in \mathbb Z$, and in that case $\deg(\Psi)$ corresponds to the usual winding number for loops in $\mathbb S^1$.

Our precise description of regularity and rigidity properties of the differential inclusion into $\Pi$ depends on the value  of this winding number.
\color{black}

\begin{thm}\label{t:NE}
Let $\Pi\subset\R^{2\times 2}$ be a closed  $C^2$ curve of length $2\pi$, without rank-one connections, nowhere elliptic, 
{\color{black}
and satisfying the nondegeneracy estimate \eqref{eq:det4}.}
Denote by $\gamma\in C^2(\R/2\pi\Z;\R^{2\times 2})$ an arc-length parametrization of $\Pi$.

For any open set $\Omega\subset\R^2$ and weakly differentiable map $u\colon\Omega\to\R^2$, if $u$ solves the differential inclusion
\begin{align*}
Du\in\Pi=\gamma(\R/2\pi\Z)\qquad\text{a.e. in }\Omega,
\end{align*}
then $Du$ is locally Lipschitz outside a locally finite set 
$\mathcal S\subset\Omega$. 

\medskip

Moreover,
letting $\Psi\colon\mathbb S^1\to\mathbb {RP}^1=\mathbb S^1/\lbrace\pm 1\rbrace$
be the $C^1$ map such that 
the image of the rank-one matrix $(\mathrm{cof}\,\gamma'(\theta))^T$ is spanned by $\Psi(e^{i\theta})$,
we have that
 $\mathcal S$ is empty if 
$|\deg(\Psi)|\notin \lbrace 1/2,1\rbrace$, and singularities are rigid if 
$|\deg(\Psi)|\in\lbrace 1/2,1\rbrace$.

\medskip

More precisely:
\begin{itemize}

\item[(a)] If $|\deg(\Psi)| \notin\lbrace 1/2,1\rbrace $, 
then $\mathcal S=\emptyset$ and
 $Du=\gamma(\theta)$ is constant along characteristic lines directed by $\Psi(e^{i\theta})$.

\item[(b)] If $|\deg(\Psi)|=1$, then  
the map $\Psi$ 
can be lifted to
 a $C^1$ diffeomorphism $\Psi\colon\mathbb S^1\to\mathbb S^1$. 
Moreover, in any convex subset $U \subset\Omega$ containing 
a singular point $x_0\in U\cap\mathcal S$,
 we have
$Du=\gamma(\theta)$ with $e^{i\theta}=\Psi^{-1}(v)$ and $v\colon U\to\mathbb S^1$ is given by
\begin{align*}
v(x)=
\tau\frac{x-x_0}{|x-x_0|}
\qquad\text{for a.e. }x\in U,
\end{align*}
for some $\tau\in\lbrace \pm 1\rbrace$.

\item[(c)] If $|\deg(\Psi)|=1/2$, then the map $\Psi\colon \mathbb S^1\to \mathbb{RP}^1=\mathbb S^1/\lbrace \pm 1 \rbrace$
is a $C^1$ diffeomorphism. Moreover, for any disk $B_{2r}(x_0) \subset\Omega$ centered at 
a singular point $\lbrace x_0\rbrace = B_{2r}(x_0)\cap\mathcal S$,
 we have
$Du=\gamma(\theta)$ with $e^{i\theta}=\Psi^{-1}(\lbrace\pm v\rbrace)$ and   $v\colon B_r(x_0)\to \mathbb S^1$ is given  by either
\begin{align*}
v(x)=
\frac{x-x_0}{|x-x_0|}
\qquad\text{for a.e. }x\in B_r(x_0),
\end{align*}
or there exists $\zeta\in\mathbb S^1$ such that
\begin{align*}
v(x)
&
\begin{cases}
=\frac{x-x_0}{|x-x_0|}
\qquad\text{for a.e. }x\in B_r(x_0)\cap\lbrace (x-x_0)\cdot\zeta >0\rbrace,\\
\text{ is Lipschitz in }B_r(x_0)\cap\lbrace (x-x_0)\cdot\zeta \leq 0\rbrace.
\end{cases}
\end{align*}
\end{itemize}

\end{thm}

\begin{rem}\label{r:nondegen_NE}
For nowhere elliptic curves, the nondegeneracy estimate \eqref{eq:det4} happens to be equivalent to the condition $\det(\gamma''(\theta))\neq 0$, see \textsection~\ref{s:det4NE}.
\end{rem}

\begin{rem}\label{r:optimal}
The rigid singularities in parts $(b)$ and $(c)$ of Theorem~\ref{t:NE} correspond to \emph{zero-energy states} of Aviles-Giga functionals.
In case $(b)$, the   $\mathbb S^1$-valued map $w=i\Psi(e^{i\theta})$ is a zero-energy state of the Aviles-Giga functional as described in \cite{JOP02}.
In case $(c)$, the $\mathbb{RP}^1$-valued map $v=i\Psi(e^{i\theta})$ is a zero-energy state of an unoriented Aviles-Giga functional, as described in \cite{GMPS23}.
In particular, optimality of our regularity statements follows from the optimality of the regularity statements in \cite{JOP02,GMPS23}.
It is also instructive
 to compare Theorem~\ref{t:NE} with \cite{Iq00}, where, given any non-elliptic curve,  nontrivial laminate measures are constructed in an arbitrarily small neighborhood of the curve. 
\end{rem}

\subsection{Nowhere elliptic curves with boundary}

Now we consider a compact connected nowhere elliptic curve with boundary. 
Since the curve is not closed, the differential inclusion \eqref{eq:diffinc} cannot have singularities with nontrivial winding numbers as in parts $(b)$ and $(c)$ of Theorem~\ref{t:NE}.

\begin{thm}\label{t:NE_bdry}
Let $\Pi\subset\R^{2\times 2}$ be a compact connected  $C^2$ curve 
with non-empty boundary, without rank-one connections, nowhere elliptic,
{\color{black}
and satisfying the nondegeneracy estimate \eqref{eq:det4}.}
Denote by $\gamma\in C^2([a,b];\R^{2\times 2})$ an arc-length parametrization of $\Pi$.

For any open set $\Omega\subset\R^2$ and weakly differentiable map $u\colon\Omega\to\R^2$, if $u$ solves the differential inclusion
\begin{align*}
Du\in\Pi=\gamma([a,b])\qquad\text{a.e. in }\Omega,
\end{align*}
then $Du=\gamma(\theta)$ is locally Lipschitz in $\Omega$,
and constant along characteristic lines directed by $\Psi(e^{i\theta})$,
where $\Psi\colon [a,b]\to \mathbb{RP}^1$ is the $C^1$ map such that
the image of the rank-one matrix $(\mathrm{cof}\,\gamma'(\theta))^T$ is spanned by $\Psi(e^{i\theta})$.
\end{thm}

The proof of Theorem~\ref{t:NE_bdry} is essentially a reproduction of the proof of Theorem~\ref{t:NE} case $(a)$. 
{\color{black}
This case with boundary is crucial in the proof of Theorem~\ref{t:reg_gen} in the partially elliptic case, via Theorem~\ref{t:PE}-$(ii.b)$. }

\subsection{Examples of closed nowhere elliptic curves}
\label{ss:ex_intro}

It is natural to wonder whether there exist many curves $\Pi\subset\R^{2\times 2}$ satisfying the assumptions
of Theorem~\ref{t:NE}.
One important example is the curve $K_0$
 studied in \cite{LLP20}, 
which is parametrized by
\begin{align*}
\gamma_2(t)=\frac 12 [e^{it}]_c +\frac{1}{6}[e^{3it}]_a.
\end{align*}
\color{black}
where, for $z\in \mathbb C$, $[z]_c$ and $[z]_a$ are the naturally associated conformal and anticonformal matrices,
\begin{align}\label{eq:zac}
[z]_c 
=
\left(\begin{array}{cc}
\mathfrak{Re}\,z  & -\mathfrak{Im}\,z \\
\mathfrak{Im}\,z & \mathfrak{Re}\,z 
\end{array}\right),
\quad [z]_a=
\left(\begin{array}{cc}
\mathfrak{Re}\,z  & \mathfrak{Im}\,z \\
\mathfrak{Im}\,z & -\mathfrak{Re}\,z 
\end{array}\right).
\end{align}
\color{black}
Indeed there is a family of related examples, parametrized by
\begin{equation}\label{eq:gamma_k}
\gamma_{k}(t)=\frac{1}{2} [e^{it}]_c +\frac{1}{2(k+1)}[e^{(k+1)it}]_a \qquad \forall k\in \mathbb{N},\, k\geq 1.
\end{equation}
That these satisfy the assumptions will be checked in \S~\ref{ss:gammak}, and 
 similar examples will be given in \S~\ref{ss:ex}.
For this curve $\gamma_k$, 
the map $\Psi$ appearing in Theorem~\ref{t:NE}
{\color{black}
is given by $\Psi(e^{i\theta})=\lbrace \pm ie^{ik\theta/2}\rbrace$,
as can be inferred from the proof of Lemma~\ref{l:lambdaPsi},
hence it}  has winding number $\deg(\Psi)=k/2$.

In general, it is easy to check whether a closed $C^2$ curve $\Pi$ parametrized by $\gamma\in C^2(\R/2\pi\Z;\R^{2\times 2})$ is nowhere elliptic, as this simply amounts to the local condition $\det(\gamma')=0$.
It is also easy to check the nondegeneracy assumption \eqref{eq:det4} as it is equivalent to $|\det(\gamma'')| > 0$, see Remark~\ref{r:nondegen_NE} and \textsection~\ref{s:det4NE}.
Let us denote by $\mathrm{NE}_*$ the set of such parametrizations:
\begin{align*}
\mathrm{NE}_*=\left\lbrace \gamma\in C^2(\R/2\pi\Z;\R^{2\times 2}) \colon \det(\gamma')= 0
\text{ and }|\det(\gamma'')|>0
\right\rbrace.
\end{align*}
What is usually harder is to check the condition that $\Pi=\gamma(\R/2\pi\Z)$ has no rank-one connections.
We show however that the subset $\mathrm{NE}_{**}\subset \mathrm{NE}_*$ which corresponds to curves without rank-one connections is somewhat large,
in the sense that it is open.

\begin{prop}\label{p:NE}
The set 
\begin{align*}
\mathrm{NE}_{**}=
\Big\lbrace \gamma\in \mathrm{NE}_*\colon
&
\gamma(\R/2\pi\Z)\text{ has no rank-one connections}
\Big\rbrace,
\end{align*}
is open in $\mathrm{NE}_*$ for the $C^2$ topology.
In particular it contains 
a neighborhood of each curve $\gamma_{k}$ \eqref{eq:gamma_k}.
\end{prop}

\subsection{Strategy of proof:  entropy productions}
\label{ss:strategy}

Our strategy to prove Theorem~\ref{t:reg_gen} is to obtain a large family of nonlinear conservation laws, which can then be used to deduce regularity and rigidity. 

The basic principle is as follows. 
The differential inclusion $Du\in\Pi=\gamma(I)$ implies $Du=\gamma(\theta)$ for some real-valued function $\theta$.
The standard identity $\na\cdot\mathrm{cof}\, Du=0$ can be interpreted as two conservation laws for the function $\theta$, 
one from each row of the matrix $\mathrm{cof}\, Du=\mathrm{cof}\,\gamma(\theta)$.
If $\theta$ is a smooth function, then the chain rule provides an infinite family of conservation laws $\na\cdot \Xi(\theta)=0$, called entropy productions in analogy with the theory of scalar conservation laws. 
(Any smooth map $\Xi\colon\R\to\R^2$ such that $\Xi'$ is a linear combination of the two rows of $\mathrm{cof}\,\gamma'$ has this property.)
If $\theta$ is not smooth, entropy productions are distributions which can in general not be computed via the chain rule, and may not vanish.
One might however expect a partial converse statement: if all entropy
 productions vanish, then $\theta$ is (somewhat) regular. 
 This is the type of regularity property on which Theorem~\ref{t:reg_gen} relies.

\color{black}
For uniformly convex
\color{black}
 scalar conservation laws, such regularity property is a consequence of regularity features of entropy solutions \cite{kruzhkov70}: if all entropy productions vanish, then the solution is
 \color{black}
 locally
 \color{black} Lipschitz, 
 \color{black}
 see Example~\ref{ex:burgers}.
 \color{black}

In the theory of the Aviles-Giga functional, methods based on entropy productions were introduced in \cite{DKMO01} to obtain compactness properties, inspired by similar arguments for scalar conservation laws \cite{tartar79,tartar82}, and have been widely used since then (see e.g. \cite{DLO03,lorent14,DLI15,GL20}).
The analog of the above regularity property is 
 the characterization by Jabin, Otto and Perthame  \cite{JOP02} of zero-energy solutions to the two-dimensional eikonal equation: if $\na\cdot m=0$ and $|m|=1$ in $\Omega\subset \R^2$, and all entropy productions vanish, then $m$ is locally Lipschitz outside a discrete set, and moreover singularities are rigid.
 This result was improved by the last two authors in \cite{LP18}, where it was shown that only two specific entropy productions $\na\cdot\Sigma_1(m)=\na\cdot\Sigma_2(m)=0$ are needed to obtain the same conclusion.
And we improved it further  in \cite{LLP20}
by showing that the original conservation law $\na\cdot m=0$ is not even needed. 
This amounts to a regularity and rigidity result for the differential inclusion
$Du\in K_0$
\color{black}
mentioned in Example~\ref{ex:AG}, 
\color{black}
where $K_0=\Sigma^\perp(\mathbb S^1)\subset\R^{2\times 2}$, and $\Sigma^\perp$ is the matrix-valued map whose two rows are $i\Sigma_1,i\Sigma_2$.
Here we identify $\R^2\approx \C$ and multiplication by $i$ corresponds to rotation by $\pi/2$.

What we explain next is that differential inclusions into curves in $\R^{2\times 2}$ 
are endowed with a structure similar to entropy productions of the two-dimensional eikonal equation $\na\cdot m=0$, $|m|=1$, 
and this structure can be used to obtain analogs of \cite{JOP02,LP18,LLP20}.

Recall that $\gamma\colon I\to\R^{2\times 2}$ is an arc-length parametrization of $\Pi$, and consider $J=\exp(iI)\subset\mathbb S^1$, that is, $J=\mathbb S^1$ if $I=\R/2\pi\Z$ and $J=\lbrace e^{i\theta}\rbrace_{\theta\in [a,b]}$ if $I=[a,b]$.
We let $\Gamma\colon J\to\R^{2\times 2}$ denote the cofactor matrix 
\begin{align}\label{eq:Gamma}
\Gamma(e^{i\theta})=\mathrm{cof}\,\gamma(\theta)\qquad\forall\theta\in I.
\end{align}
Any solution of the differential inclusion \eqref{eq:diffinc} satisfies $\mathrm{cof}\, Du \in\Gamma(J)$ a.e., so there exists $m\colon  \Omega\to J\subset\mathbb S^1$ such that
\begin{align}\label{eq:cofDuGammam}
\mathrm{cof}\, Du =\Gamma(m).
\end{align}
Therefore the identity
$\na\cdot \mathrm{cof}\, Du=0$ implies
\begin{align}
\label{eq:divGamma0}
\nabla\cdot \Gamma_1(m)=\nabla\cdot\Gamma_2(m)=0\qquad\text{in }\mathcal D'(\Omega),
\end{align}
where $\Gamma_1,\Gamma_2\colon J\to\R^2$ are the first and second rows of the matrix-valued map $\Gamma$. 
So the unit vector field $m$ satisfies two conservation laws,
 similarly to the examples~\ref{ex:burgers} and \ref{ex:AG}.
If $m$ were smooth, then the chain rule 
would provide an infinite family of conservation laws
\begin{align*}
\nabla\cdot \Phi(m)=0\qquad\text{in }\mathcal D'(\Omega),
\end{align*}
for any smooth map $\Phi\colon J\to\R^2$ such that 
\begin{align*}
\partial_\theta\Phi =\alpha_\Phi^1 \partial_\theta\Gamma_1 +\alpha_\Phi^2 \partial_\theta\Gamma_2,
\qquad
\alpha_\Phi^1,\alpha_\Phi^2\colon J\to\R.
\end{align*}
In analogy with scalar conservation laws, we follow the terminology of \cite{DKMO01} and call such maps $\Phi$ entropies, the distributions $\nabla\cdot\Phi(m)$ entropy productions, and we let
\begin{align}\label{eq:ENT}
\mathrm{ENT}_\Gamma =\Big\lbrace
\Phi\in C^2(J;\R^2)
\colon
&
\exists \alpha_\Phi^1,\alpha_\Phi^2\in C^1(J;\R)\text{ s.t. }
\nonumber
\\
&
\partial_\theta\Phi =\alpha_\Phi^1 \partial_\theta\Gamma_1 +\alpha_\Phi^2 \partial_\theta\Gamma_2
\Big\rbrace.
\end{align} 
Our map $m$ is not smooth enough to apply the chain rule, so it is not obvious that entropy productions should vanish.
If we manage to prove that they do, then the ideas of \cite{JOP02,GMPS23} can be applied to obtain regularity and rigidity (\color{black}see \S~\ref{s:proofNE}\color{black}).
Therefore, the main ingredient in our proof of Theorem~\ref{t:reg_gen} is the following proposition, which shows that  entropy productions do vanish provided $m$ has some low fractional regularity 
(it is classical that this starting low regularity is guaranteed by the quartic nondegeneracy estimate \eqref{eq:det4}, see Lemma~\ref{l:besov}).
We state it here in the nowhere elliptic
 case where $\det(\partial_\theta\Gamma) \equiv 0$ 
 (i.e. all tangent lines to $\Pi$ are generated by rank-one matrices),
under the nondegeneracy assumption that $\det(\partial^2_\theta\Gamma)$ does not vanish, which  in that case is equivalent to
the quartic estimate
 \eqref{eq:det4} (see Lemma~\ref{l:detcontrols4}).
It will a posteriori be valid in the general setting of Theorem~\ref{t:reg_gen}.

\begin{prop}\label{p:conslaws}
Let $J\subset\mathbb S^1$ be compact and connected.
Assume that $\Gamma\in C^2(J;\R^{2\times 2})$ satisfies $|\partial_\theta\Gamma|=1$, $\det(\partial_\theta\Gamma)=0$ and $|\det(\partial_\theta^2\Gamma)|>0$  on $J$.
	Then any solution $m\colon \Omega\to J\subset\mathbb S^1$ of \eqref{eq:divGamma0} satisfies
	\begin{align}\label{eq:conslaws}
		\nabla\cdot\Phi(m)=0\quad\text{ in }\mathcal D'(\Omega),\quad\forall\Phi\in \mathrm{ENT}_\Gamma,
	\end{align}
	provided $m\in B^{\frac 13}_{p,\infty,\mathrm{loc}}(\Omega;\mathbb S^1)$ for some $p>3$, that is,
	\begin{align}\label{eq:besov}
		\sup_{|h|\leq 1}\frac{\|D^h m\|_{L^p(U)}}{|h|^{\frac 13}}<\infty,\qquad 
		D^h m(x)= (m(x+h)-m(x))\mathbf 1_{x,x+h\in\Omega}\, ,
	\end{align}
	for all $U\subset\subset\Omega$.
\end{prop}

A less general version of the above proposition also lies at the heart of our previous work \cite{LLP20},
 in a special case where $\partial_\theta \Gamma(e^{i\theta})=\lambda(e^{i\theta})\otimes ie^{i\theta}$ for some $\lambda\colon\mathbb S^1\to\mathbb S^1$, and the identity mapping is an entropy:
 $\mathrm{id}_{\mathbb S^1}\in \mathrm{ENT}_\Gamma$.
Here our proof follows a similar road map, but requires new ingredients to deal with  the more general setting.

The basic principle is as follows. 
In order to apply the chain rule, we consider a mollified map $m_\e$, 
but this destroys the nonconvex constraint $|m|=1$:
the identities showing that  entropy productions vanish for a smooth $m$ are not valid for $m_\e$. 
Our task consists in proving that the error terms thus introduced are negligible as $\e\to 0$.
A now standard tool for these kinds of problems is the commutator estimates of  \cite{CET94}.
Here they
play an important role, 
but,
 contrary to \cite{DLI15},
  they are far from enough to conclude directly: 
they only serve to show that
entropy productions are in $L^{p/3}_{\mathrm{loc}}$.
As in \cite{LLP20}, we then need to bootstrap that information into eventually obtaining that entropy productions vanish.
A crucial feature in \cite{LLP20} was the special role played by the identity mapping: an entropy that can be extended to a linear mapping of $\R^2$.
Here we do not have this structure in general and rely on a different argument.
A well-designed decomposition and careful use of commutator estimates
 enable us to obtain  identities relating entropy productions and weak limits of some error terms, see \eqref{eq:dvPhif1f2}.
 Testing these identities with well-chosen entropies 
shows that the error terms vanish.

\begin{rem} 
A key observation in \cite{LLP20}
was that different choices of extensions of entropies might provide different information in the limit, but
our new argument allows us to use only classical radial extension, and  provides a simpler proof of the main result in \cite{LLP20}.
\end{rem}

\subsection{Plan of the article}

\color{black}
In \textsection~\ref{s:besov} we establish the initial low fractional regularity which follows from the nondegeneracy estimate \eqref{eq:det4}.
In \textsection~\ref{s:proof_PE} we give the proof of Theorem~\ref{t:PE} reducing our main result to the nowhere elliptic case.
Then we focus on nowhere elliptic curves.
In \textsection~\ref{s:det4NE} we establish the equivalent characterization of the nondegeneracy estimate \eqref{eq:det4} mentioned in Remark~\ref{r:nondegen_NE}.
In \textsection~\ref{s:Psi} we give an explicit expression of the map $\Psi$ defined in \textsection~\ref{ss:NE_closed_intro}.
In \textsection~\ref{s:conslaws} we prove Proposition~\ref{p:conslaws}.
In \textsection~\ref{s:proofNE} we give the 
proofs of Theorems~\ref{t:NE} and \ref{t:NE_bdry}.
In \textsection~\ref{s:ex_NE} we provide examples of closed nowhere elliptic curves and the proof of Proposition~\ref{p:NE}.
In Appendix~\ref{a:rig} we give the extension, to curves with boundaries, of a rigidity estimate proved in \cite{LLP23} for closed curves.
\color{black}

\bigskip
\noindent \bf  Acknowledgments. \rm

We thank Riccardo Tione for showing us the short proof of Proposition~\ref{p:reg_ellipt} that we present here, and for several fruitful discussions.
{\color{black}
We thank the anonymous referee for suggesting many substantial improvements in the structure of the article and pointing out \cite[Theorem~6.5]{BM21}.
}
 XL received support from ANR project ANR-22-CE40-0006. 
AL gratefully acknowledges the support of the Simons foundation, collaboration grant \#426900. GP was supported in part by NSF grant DMS-2206291. Part of this work
was conducted during XL’s stays at  the Hausdorff Institute for Mathematics (HIM) in Bonn,
funded by the Deutsche Forschungsgemeinschaft (DFG, German Research Foundation)
under Germany’s Excellence Strategy – EXC-2047/1 – 390685813, during the Trimester Program “Mathematics for Complex Materials”.

\section{Fractional regularity of $m$} 
\label{s:besov}
 
In this section we establish the initial low fractional regularity which is required to apply Proposition~\ref{p:conslaws}.
Since $\Pi$ has no rank-one connections, \cite[Lemma 1]{sverak93} implies
 that
 $\det(A-B)$ has a constant sign over all matrices $ A\neq B \in \Pi$. 
Composing $u$ with a reflection does not affect Theorem~\ref{t:reg_gen}, so we may assume without loss of generality
\begin{align*}
\det(A-B)>0\qquad\forall A\neq B\in\Pi,
\end{align*}
and the coercivity of the determinant \eqref{eq:det4} can be rephrased as
\begin{align*}
\det(A-B)\geq c \, |A-B|^4 \qquad\forall A, B\in\Pi.
\end{align*}
Quite classically,
this
implies some fractional regularity for the map $m$ 
defined in \eqref{eq:cofDuGammam}, 
which satisfies $Du=\mathrm{cof}\,\Gamma(m)$. 
That fractional regularity follows from 
directly adapting the proof of \cite[Theorem~5]{sverak93}.
This adaptation has already been done in \cite[Proposition~1]{FK},
where the conclusion is written in terms of Sobolev regularity.
For the reader's convenience we include here the statement in
terms of Besov regularity and reproduce its proof.

\begin{lem}\label{l:besov}
Let $J\subset\mathbb S^1$ be compact and connected. If 
$m\colon\Omega\to J$
 solves \eqref{eq:divGamma0} where $\Gamma\in C^1(J;\R^{2\times 2})$  satisfies
\begin{align}\label{eq:pcoercivity}
\det(\Gamma(z)-\Gamma(z'))\geq c |z-z'|^p\qquad\forall z,z'\in J,
\end{align}
for some $p> 1$ and $c>0$,
then $m\in B^{\frac{1}{p-1}}_{p,\infty,\mathrm{loc}}(\Omega;\mathbb S^1)$, that is,
\begin{align*}
\sup_{|h|\leq 1}\frac{\|D^h m\|_{L^p(U)}}{|h|^{\frac 1{p-1}}}<\infty\qquad\text{for all }U\subset\subset\Omega,
\end{align*}
where $D^hm(x)=(m(x+h)-m(x))\mathbf 1_{x,x+h\in\Omega}$\ .
In particular, if $p=4$ then $m\in B^{\frac{1}{3}}_{4,\infty,\mathrm{loc}}$.
\end{lem}

\begin{proof}[Proof of Lemma~\ref{l:besov}]
The proof 
is essentially the same as the first step in \cite[Theorem~5]{sverak93}.

Since this is a local result, we may assume without loss of generality that $\Omega$ is simply connected.
 Since $\na\cdot\Gamma_j(m)=0$, we infer that $\mathrm{curl}\lt(i\Gamma_j(m)\rt)=0$, and thus there exists $F_j:\Omega\rightarrow \R$ with 
	\begin{equation}\label{eq13}
		\nabla F_j=i\Gamma_j(m)\qquad\text{a.e. in }\Omega. 
	\end{equation}
	For any given $U\subset\subset\Omega$ and $h\in\R^2$ with $|h|$ sufficiently small, e.g. $|h|<\frac{1}{3}\mathrm{dist}(U,\partial\Omega)$, by \eqref{eq:pcoercivity} we have
	\begin{equation*}
		\det \lt(\Gamma(m(x+h))-\Gamma(m(x))\rt)\geq c|D^hm(x)|^p
	\end{equation*}
	for a.e. $x\in\Omega$ with $\dist(x,\partial\Omega)>|h|$. By \eqref{eq13} we have
	\begin{equation*}
		\det \lt(\Gamma(m(\cdot +h))-\Gamma(m(\cdot))\rt) = (iD^h\Gamma_1(m))\cdot ( D^h\Gamma_2(m))=D^h \nabla F_1\cdot D^h\Gamma_2(m).
	\end{equation*}
	Hence gathering the two above equations, we obtain
	\begin{equation}\label{eq:detcontrolsDhm}
		\abs{D^h m}^p\lesssim D^h \nabla F_1\cdot D^h\Gamma_2(m) \qquad\text{ for a.e. }x\in\Omega\text{ with }\mathrm{dist}(x,\partial\Omega)>|h|.
	\end{equation}
	Let $\chi\in C_c^\infty(\Omega)$ be a test function with 
	$\mathrm{dist}(\supp\chi,\partial\Omega)>2h$ and 
	$\mathbf 1_U\leq\chi\leq\mathbf 1_{\Omega}$. 
	Integrating by parts and using that $\na\cdot\Gamma_2(m)=0$ (and thus $\int_{\Omega} \na \lt(D^h F_1 \chi^{\frac{p}{p-1}} \rt)\cdot D^h \Gamma_2(m) dx=0$), we have
	\begin{align*}
		&\int_\Omega \chi^{\frac{p}{p-1}} D^h \nabla F_1 \cdot D^h\Gamma_2(m) \, dx
		=- \int_\Omega D^h F_1\,  D^h\Gamma_2(m) \cdot \nabla (\chi^{\frac{p}{p-1}})\, dx\\
		&\qd\qd\qd\qd\qd\qd\qd\qd\lesssim \abs{h}\,\norm{\na F_1}_{L^\infty(\Omega)} \|\partial_\theta \Gamma_2\|_{L^{\infty}(J)} \norm{\nabla\chi}_{L^\infty(\Omega)} \int_\Omega \chi^{\frac{1}{p-1}}\abs{D^h m}\,dx \\
		&\qd\qd\qd\qd\qd\qd\qd\qd\lesssim \abs{h}\,\norm{\nabla\chi}_{L^\infty(\Omega)} \left(\int_\Omega \chi^{\frac{p}{p-1}} \abs{D^h m}^p \, dx\right)^{\frac 1p}.
	\end{align*}
	Recalling \eqref{eq:detcontrolsDhm} we deduce
	\begin{equation*}
		\int_\Omega \chi^{\frac{p}{p-1}}\abs{D^h m}^p\, dx \lesssim \abs{h}\,\norm{\nabla\chi}_{L^\infty(\Omega)} \left(\int_\Omega \chi^{\frac{p}{p-1}} \abs{D^h m}^p \, dx\right)^{\frac 1p},
	\end{equation*}
	and thus
	\begin{equation*}
		\left(\int_\Omega \chi^{\frac{p}{p-1}} \abs{D^h m}^p \, dx\right)^{\frac 1p}\lesssim \abs{h}^{\frac{1}{p-1}}\,\norm{\nabla\chi}_{L^\infty(\Omega)}^{\frac{1}{p-1}}.
	\end{equation*}
	As $\mathbf 1_U\leq\chi$,  it follows that
	\begin{equation*}
		\sup_{\abs{h}\leq t}\frac{\norm{D^h m}_{L^p(U)}}{|h|^{\frac{1}{p-1}}}\lesssim\norm{\nabla\chi}_{L^\infty(\Omega)}^{\frac{1}{p-1}}
	\end{equation*}
	for $t>0$ sufficiently small. For larger $|h|$ values, the boundedness of $m$ implies  that $\sup_{t<\abs{h}\leq 1}|h|^{-\frac{1}{p-1}}\norm{D^h m}_{L^p(U)}$ is bounded. Thus $m\in B^{\frac{1}{p-1}}_{p,\infty}(U)$ for all $U\subset\subset\Omega$.
\end{proof}

 \section{\color{black}Reduction to the nowhere elliptic case: proof of Theorem~\ref{t:PE}}
 \label{s:proof_PE}

\subsection{Regularity at elliptic values}
\label{ss:regE}

In this section we prove part $(i)$ of Theorem~\ref{t:PE}, 
that 
$Du\in\Pi_{NE}$ or $Du$ is constant.
For this part we do not need to assume the quartic estimate \eqref{eq:det4}, but only the fact that $\Pi$ has no rank-one connections. 
Recall from the introduction that we
assume without loss of generality that $\Pi$ has length 
at most $2\pi$, and $\gamma\colon  I \to\Pi$ is an arc-length parametrization of class $C^2$,
where $I$ is either a segment $[a,b]\subset\R$ ($a<b<a+2\pi$) in the case with boundary, or $I=\R/2\pi\Z$ in the case without boundary.

\begin{prop}\label{p:reg_ellipt}
Let $\Pi\subset\R^{2\times 2}$ be a  
compact connected  $C^2$ curve of length 
at most $2\pi$, without rank-one connections. Denote by $\gamma\in C^2(I;\R^{2\times 2})$ an arc-length parametrization of $\Pi$,
and by $\Pi_E=\gamma(\lbrace \det(\gamma')\neq 0\rbrace)\subset\Pi$ the subset of elliptic values.
For any connected open set $\Omega\subset\R^2$ and weakly differentiable map $u\colon\Omega\to\R^2$, if $u$ solves the differential inclusion
\begin{align*}
Du\in\Pi=\gamma(I)\qquad\text{a.e. in }\Omega,
\end{align*}
and $\Omega_E=(Du)^{-1}(\Pi_E)$ has positive Lebesgue measure, then 
 $Du$ is constant.
\end{prop}

\begin{proof}[Proof of Proposition~\ref{p:reg_ellipt}]
The proof we present here, 
based on the notion of quasiconformal envelope \cite{FS08,KS08}
and a recent unique continuation result of \cite{DPGT23}, 
has been pointed out to us by Riccardo Tione.
(Another possible proof, based on \cite[Theorem~1.9]{lledos23}, 
would require a refined, and rather technical to demonstrate, version of Kirszbraun's extension theorem.)

In a nutshell, the argument is as follows.
First, 
the curve $\Pi$ is included in the quasiconformal envelope $\mathcal E_{\mathcal J}$ of any nontrivial
compact elliptic arc $\mathcal J\subset\Pi_{E}$, 
that is, there exists $K>0$ such that
\begin{align}\label{eq:PiquasiconvA0}
|A-B|^2 \leq K \det (A-B)\qquad\forall A\in \mathcal J\text{ and }B\in\Pi.
\end{align}
Second, \color{black}due to \cite{LLP23} and its adaption in Proposition~\ref{p:rig_est_J}, we know that
\color{black}
such elliptic arc satisfies a quantitative rigidity property as in \cite{FJM02}, that is,
there exists $c>0$ such that 
\begin{align}\label{eq:rigid_est_J}
\inf_{A\in \mathcal J}\int_{B_{1/2}} |Du-A|^2\, dx \leq c \int_{B_1} \dist^2(Du,\mathcal J)\, dx
\qquad\forall u\in W^{1,2}(B_1;\R^2).
\end{align}
We can therefore directly invoke \cite[Theorem~1.3]{DPGT23} which
ensures that $\dist(Du,\mathcal J)$ is either positive a.e. or identically zero.

Now we proceed to give a more detailed proof.
As in the beginning of \textsection~\ref{s:besov} we assume without loss of generality that
\begin{align}\label{eq:detposit}
\det(A-B)>0\qquad\forall A\neq B \in \Pi.
\end{align}
We assume that $\Omega_E$ has positive Lebesgue measure: then we can pick a Lebesgue point $x_0$ of $Du$ such that $Du(x_0)\in \Pi_E$.
Since $\Pi_E$ is open in $\Pi$, 
for every neighborhood of $Du(x_0)$ in $\Pi_E$, its preimage
 by $Du$ has positive Lebesgue measure. 
Therefore we may fix $a'<b'<a'+2\pi$ such that 
$\det(\gamma')\neq 0$ on $[a',b']\subset I$ and, letting 
\begin{align*}
\mathcal J=\gamma([a',b'])\subset\Pi_E,
\end{align*}
we have that 
\begin{align}\label{eq:distJnonzero}
\dist(Du,\mathcal J)^{-1}(\lbrace 0\rbrace)\text{ has positive Lebesgue measure.}
\end{align}
For every $t_0\in [a',b']$ and $h\to 0$, we have
\begin{align*}
\det(\gamma(t_0+h)-\gamma(t_0))
=h^2 \det(\gamma'(t_0)) +o(h^2).
\end{align*}
According to \eqref{eq:detposit} this implies in particular  $\det(\gamma')>0$ on $[a',b']$, 
and by compactness of $[a',b']$ we deduce the existence of $\delta_0>0$ and $K_0>0$ such that
\begin{align*}
|A-B|^2 \leq K_0 \det (A-B)\qquad\forall A\in \mathcal J\text{ and }B\in\Pi \text{ s.t. }|A-B|\leq \delta_0.
\end{align*}
Thanks to \eqref{eq:detposit} we have
\begin{align*}
c_0 :=\min \left\lbrace \frac{\det(A-B)}{|A-B|^2}\colon A\in \mathcal J,\,B\in\Pi,\, |A-B|\geq \delta_0\right\rbrace >0,
\end{align*}
and setting $K=\max(K_0,1/c_0)$ we deduce 
the inequality \eqref{eq:PiquasiconvA0}.
In the terminology of \cite{FS08,KS08}, this means that 
the curve $\Pi$ is included in the $K$-quasiconformal envelope $\mathcal E_{\mathcal J}$ 
 of the arc $\mathcal J \subset \R^{2\times 2}$.
Here, as in \cite{DPGT23}, we omit the dependence on the fixed $K>0$ from the notation.
Therefore the differential inclusion $Du\in\Pi$ implies
\begin{align}\label{eq:diffinc_quasiconvA0}
Du\in \mathcal E_{\mathcal J}\qquad\text{a.e. in }\Omega.
\end{align}
Moreover, the inequality \eqref{eq:PiquasiconvA0} for $A,B\in \mathcal J$ 
means that 
$\mathcal J$ is \emph{elliptic} 
 and satisfies therefore 
 the \emph{rigidity estimate}
 \eqref{eq:rigid_est_J}.
This follows from a minor adaptation of \cite{LLP23} presented in Appendix~\ref{a:rig}.
As a consequence, we can apply the unique continuation principle of \cite[Theorem~1.3]{DPGT23} to the differential inclusion \eqref{eq:diffinc_quasiconvA0}, and deduce that either $\dist(Du,\mathcal J)>0$ a.e. in $\Omega$, or $\dist(Du,\mathcal J)=0$ a.e. in $\Omega$ and $Du$ is constant.
The first possibility is ruled out by \eqref{eq:distJnonzero}, so we conclude that $Du$ is constant equal to $Du(x_0)$.
\end{proof}

\subsection{Improvement under low fractional regularity}

\color{black}
In this section we give the short proof of part $(ii)$ of Theorem~\ref{t:PE},
which states that part $(i)$ can be improved to 
$Du$ being either constant or with values in a single connected component of $\Pi_{NE}=\gamma(\lbrace \det(\gamma')=0\rbrace)$,
under the nondegeneracy assumption \eqref{eq:det4}.

Thanks to Lemma~\ref{l:besov}, the assumption \eqref{eq:det4}
ensures indeed that $Du=\mathrm{cof}\:\Gamma(m)$ with 
$m\in B^{1/3}_{4,\infty,\mathrm{loc}}(\Omega)$.
The map $\Gamma$ being Lipschitz, we deduce that
$Du \in B^{1/3}_{4,\infty,\mathrm{loc}}(\Omega) \subset W^{2/7,4}_{\mathrm{loc}}(\Omega)$.
Since $(2/7)\cdot 4 >1$, 
this implies that the essential range of $Du$ 
(the smallest closed subset of $\R^{2\times 2}$ containing almost every value of $Du$) is connected, see e.g. \cite[Theorem~6.5]{BM21}.
As a consequence, if $Du$ is not constant,
then its essential range, 
which we know to be contained in the closed set $\Pi_{NE}$ 
 by Proposition~\ref{p:reg_ellipt},
 must be contained in a single connected component of $\Pi_{NE}$.
 \qed

\color{black}

\section{\color{black}The quartic nondegeneracy assumption \eqref{eq:det4} for nowhere elliptic curves}\label{s:det4NE}

As in the beginning of \textsection~\ref{s:besov}, we assume without loss of generality
\begin{align*}
\det(A-B)>0\qquad\forall A\neq B\in\Pi.
\end{align*}
Here we show that the nondegeneracy estimate \eqref{eq:det4} is equivalent, in the nowhere elliptic setting, to the assumption that $\det(\gamma'')$ does not vanish.

\begin{lem}\label{l:detcontrols4}
Let $I=[a,b]$ for some $a<b\leq a+2\pi$
and  $\gamma\in C^{2}(I;\R^{2\times 2})$. If   
$\det(\gamma')\equiv 0$ on $I$
 and
\begin{align}\label{eq:detpositive}
\det(\gamma(t)-\gamma(s))> 0\qquad\forall t\neq s\in I,
\end{align}
then
the estimate
\begin{align}\label{eq:detcontrols4}
\det(\gamma(t)-\gamma(s))\geq c |e^{it}-e^{is}|^4\qquad\forall t,s\in  I
\end{align}
is satisfied
if and only if $\det(\gamma'')$ does not vanish on $I$.
\end{lem}

\begin{rem}
The assumptions of Lemma~\ref{l:detcontrols4}
imply that $\det(\gamma'')<0$, see \eqref{eq:limdetover4}.
 The case $\det(\gamma'')>0$ would correspond to the  opposite estimate
	\begin{align*}
		-\det(\gamma(t)-\gamma(s))\geq c |e^{it}-e^{is}|^4\qquad\text{for some }c>0.
	\end{align*}
\end{rem}

\begin{proof}[Proof of Lemma~\ref{l:detcontrols4}]

First we prove that, for any $t\in I$,
\begin{align}\label{eq:limdetover4}
\lim_{s \to t, s\in I}\frac{\det(\gamma(s)-\gamma(t))}{(s-t)^4} =-\frac {1}{12} \det(\gamma''(t)).
\end{align}
Under the additional regularity assumption $\gamma\in C^3(I;\R^{2\times 2})$ this follows directly from a Taylor expansion.
Indeed,
derivating twice the identity $\det(\gamma')\equiv 0$, 
 we obtain
\begin{align*}
&\mathrm{cof}\,\gamma' \colon \gamma'' =0,\\
&
\mathrm{cof}\,\gamma' \colon \gamma^{(3)}=-\mathrm{cof}\,\gamma''\colon\gamma''=-2\det(\gamma''),
\end{align*}
and then 
 the Taylor expansion
\begin{align*}
&
\det(\gamma(t+h)-\gamma(t))
=\det\left(h\gamma'(t) +\frac  {h^2}{2} \gamma''(t) +\frac {h^3}{6}\gamma^{(3)}(t) + o(h^3)\right) \\
&
=
h^2\det(\gamma'(t))
+\frac{h^3}{2}\mathrm{cof}\,\gamma'(t)\colon\gamma''(t)
+\frac{h^4}{4}
\det(\gamma''(t)) + \frac{h^4}{6} \mathrm{cof}\,\gamma'(t)\colon\gamma^{(3)}(t) 
+o(h^4)\\
& =-\frac {h^4}{12}\det(\gamma''(t))+o(h^4),
\end{align*}
implies \eqref{eq:limdetover4}.

In the general case where  $\gamma\in C^2(I;\R^{2\times 2})$, 
one may actually use the bilinearity of the determinant to obtain the same outcome with only two derivatives. 
Specifically, for $t, t+h\in I$ we have
\begin{align}
\det(\gamma(t+h)-\gamma(t))
&
=\det\left(h\gamma'(t)+h^2\int_0^1(1-s)\gamma''(t+sh)\, ds\right)
\nonumber
\\
&
=h^2 \det(\gamma'(t))
+ h^3 \int_0^1 (1-s)\,
\mathrm{cof}\,\gamma'(t)\colon \gamma''(t+sh)\, ds
\nonumber
\\
&
\quad
+h^4 \det \left(\int_0^1
(1-s) \gamma''(t+sh)\, ds
\right).
\label{eq:det_incr_gamma_1}
\end{align} 
Using that $\mathrm{cof}\gamma'\colon\gamma''=(\det(\gamma'))'$, we can rewrite the penultimate term as
\begin{align*}
&
h^3\int_0^1 (1-s)\,
\mathrm{cof}\,\gamma'(t)\colon \gamma''(t+sh)\, ds
\\
&
=h^2\int_0^1 (1-s)\frac{d}{ds}[\det(\gamma'(t+sh)]\, ds
\\
&\quad
-
h^4 \int_0^1 s(1-s) \,\mathrm{cof} \left(\frac{\gamma'(t+sh)-\gamma'(t)}{sh}\right)\colon\gamma''(t+sh)\, ds
\\
&
=h^2\int_0^1\det(\gamma'(t+sh))\, ds -h^2\det(\gamma'(t)) -\frac{h^4}{3}  \det\gamma''(t) +h^4\e(t,h),
\end{align*}
where
\begin{align*}
\e(t,h)
&
=
\int_0^1 s(1-s)\,
 \Bigg[\mathrm{cof}\:\gamma''(t)\colon\gamma''(t)
 -
\mathrm{cof} \left(\frac{\gamma'(t+sh)-\gamma'(t)}{sh}\right)\colon\gamma''(t+sh)
 \Bigg]\, ds
 \\
&
=
\int_0^1\int_0^1 s(1-s)\,
 \Bigg[\mathrm{cof}\:\gamma''(t)\colon\gamma''(t)
 -
\mathrm{cof}\: \gamma''(t+\tau sh)\colon\gamma''(t+sh)
 \Bigg]\, d\tau\, ds,
\end{align*}
hence
 $\sup_t |\e(t,h)|\to 0$ as $h\to 0$ thanks to the uniform continuity of $\gamma''$.
Plugging this into \eqref{eq:det_incr_gamma_1} gives
\begin{align}\label{eq:incr_det_gamma}
\det(\gamma(t+h)-\gamma(t))
&
=h^2\int_{0}^1 \det(\gamma'(t+sh))\, ds
-\frac{h^4}{12}\det(\gamma''(t))
\nonumber
\\
&\quad
 +h^4\tilde\e(t,h),
\end{align}
where
\begin{align*}
\tilde \e(t,h)
&=
\e(t,h)
+\frac 12 \int_0^1(1-s)\:\mathrm{cof}\:\gamma''(t)\colon (\gamma''(t+sh)-\gamma''(t))\, ds
\\
&\quad
+\det\left(\int_0^1(1-s)(\gamma''(t+sh)-\gamma''(t))\, ds\right),
\end{align*}
so $\sup_t |\tilde \e(t,h)|\to 0$ as $h\to 0$ thanks again to the uniform continuity of $\gamma''$.
Since $\det(\gamma')\equiv 0$ on $I$, this implies \eqref{eq:limdetover4}.

A first consequence of \eqref{eq:limdetover4} is that the estimate \eqref{eq:detcontrols4} implies that $\det(\gamma'')$ does not vanish.

For the converse, assume that $\det(\gamma'')$ does not vanish.
Since the left-hand side of \eqref{eq:limdetover4} must be nonnegative according to \eqref{eq:detpositive}, we deduce that $\det(\gamma'')< 0$, and for all $t\in I$ there exists $\delta_t>0$ such that
\begin{align*}
\det(\gamma(t+h)-\gamma(t))\geq \frac{h^4}{24}\inf_{I} |\det(\gamma'')|
\qquad\forall h\in (-\delta_t,\delta_t), \,t+h\in I.
\end{align*}
By compactness of $I$ this $\delta_t$ can be chosen independent of $t$, so we have
\begin{align*}
\det(\gamma(t)-\gamma(s))\geq c |e^{it}-e^{is}|^4\qquad\text{if }|e^{it}-e^{is}|\leq\eta, \text{ for } s,t
\in I,
\end{align*} 
for some constants $c,\eta>0$ depending on $\gamma$.
Moreover by compactness of the set of couples $(t,s)\in I\times I$ such that $|e^{is}-e^{it}|\geq \eta$, there exists $c>0$ such that
\begin{align*}
\det(\gamma(t)-\gamma(s))\geq c \geq \frac{c}{2^4}|e^{it}-e^{is}|^4\qquad\text{if }|e^{it}-e^{is}|\geq\eta,
\end{align*}
and this proves Lemma~\ref{l:detcontrols4}.
\end{proof}

\section{Explicit form of the map $\Psi$ for nowhere elliptic curves}\label{s:Psi}

\color{black}
The different cases in the conclusion of Theorem~\ref{t:NE}
about nowhere elliptic closed curves
are due to different topological properties of the tangent bundle of the curve $\Pi$,
described by
 the map $\Psi\colon\mathbb S^1\to\mathbb{RP}^1$ defined at the start of \textsection~\ref{ss:NE_closed_intro}. 
In this section we
give a more explicit form of this map $\Psi$ (see Remark~\ref{r:Psi}), 
which will be convenient for the proofs.
\color{black}

\begin{lem}\label{l:lambdaPsi}
	Assume that $\gamma\in C^2(\R/2\pi\mathbb{Z};\R^{2\times 2})$ satisfies $|\gamma'(\theta)|=1$, $\det(\gamma'(\theta))=0$ and $\det(\gamma''(\theta))\ne 0$ for all $\theta\in\R$. Then there exist $\hat\lambda,\hat\Psi\in C^1(\R;\mathbb S^1)$, 
	\color{black}	and integers $k,\ell\in\Z$,
	such that, for all $\theta\in\R$,
	\begin{align}\label{eq:lambdaPsi}
		\mathrm{cof}\, \gamma'(\theta)
		&
		=\partial_\theta\Gamma(e^{i\theta})
		=\hat\lambda(\theta)\otimes \hat\Psi(\theta),
	\end{align}
	where $\Gamma$ is defined by \eqref{eq:Gamma},
	and
	\begin{align*}
		\hat\Psi(\theta+2\pi)
		&
		=e^{i k\pi}\hat\Psi(\theta),
		\quad
		\hat\lambda(\theta+2\pi)=e^{i \ell \pi}\hat\lambda(\theta).
		\nonumber
	\end{align*}
	\color{black}
	Moreover, $\hat\lambda,\hat\Psi$ have strictly monotone phases:
	\begin{align*}
		\hat\lambda=e^{i\varphi_\lambda},\;\hat\Psi=e^{i\varphi_\Psi},\text{ with }
		\varphi_\lambda,\varphi_\Psi\in C^1(\R;\R)
		\text{ such that }|\varphi_\lambda'|,|\varphi_\Psi'|>0.
	\end{align*}
\end{lem}

\begin{proof}[Proof of Lemma~\ref{l:lambdaPsi}]

\color{black}
Let 
$\gamma_c,\gamma_a\in C^2(\R/2\pi\Z;\C)$ 
correspond to the conformal and anticonformal parts of $\gamma$, that is,
\begin{align}
\label{eq:gamma_ac}
\gamma
&
=
[\gamma_c]_c +[\gamma_a]_a, 
\end{align}
where, for $z\in \mathbb C$, $[z]_c$ and $[z]_a$ are the associated conformal and anticonformal matrices as in \eqref{eq:zac}.
Since $\Pi$ is nowhere elliptic and $\gamma$ is an arc-length parametrization, we have
 $0=\det(\gamma') =|\gamma'_c|^2-|\gamma'_a|^2$ and $1= |\gamma'|^2=2|\gamma'_c|^2 +2 |\gamma'_a|^2$,
and deduce that
\begin{align*}
|\gamma'_c|=|\gamma'_a|=\frac 12.
\end{align*}
In particular, the maps $\gamma_c',\gamma'_a\colon\R/2\pi\Z\to\frac 12\mathbb S^1$ have  well-defined winding numbers, $\deg(\gamma_c'),\deg(\gamma_a')\in\Z$. 
\color{black}
	We may choose $\varphi_c, \varphi_a \in C^1(\R;\R)$ such that
	\begin{align*}
		\gamma_c'(\theta) = \frac 12 e^{i\varphi_c(\theta)}, \qquad \gamma_a'(\theta) = \frac 12 e^{i\varphi_a(\theta)}, 
	\end{align*}
	and  we have
	\begin{align*}
		\mathrm{cof}\,\gamma' &= \left(\begin{array}{cc}
			-\sin\lt(\frac{\varphi_c+\varphi_a}{2}\rt)\sin\lt(\frac{\varphi_c-\varphi_a}{2}\rt)  & -\sin\lt(\frac{\varphi_c+\varphi_a}{2}\rt)\cos\lt(\frac{\varphi_c-\varphi_a}{2}\rt) \\
			\cos\lt(\frac{\varphi_c+\varphi_a}{2}\rt)\sin\lt(\frac{\varphi_c-\varphi_a}{2}\rt) & \cos\lt(\frac{\varphi_c+\varphi_a}{2}\rt)\cos\lt(\frac{\varphi_c-\varphi_a}{2}\rt)
		\end{array}\right)\\
		&= i e^{i\frac{\varphi_c+\varphi_a}{2}}\otimes i e^{ i\frac{\varphi_a-\varphi_c}{2}}, 
	\end{align*}
	so we can define
	\begin{align*}
		\hat\lambda(\theta) = i e^{i\frac{\varphi_c(\theta)+\varphi_a(\theta)}{2}}, \qquad \hat\Psi(\theta) =  
		i e^{i\frac{\varphi_a(\theta)-\varphi_c(\theta)}{2}}. 
	\end{align*}
	\color{black}
	By definition of the winding numbers $\deg(\gamma_a')$, $\deg(\gamma_c')$, we have
	\begin{align*}
	(\varphi_a \pm \varphi_c)(\theta+2\pi)-(\varphi_a \pm \varphi_c)(\theta) =2\pi (\deg(\gamma_a')\pm\deg(\gamma_c')).
	\end{align*}
	Letting $k=\deg(\gamma_a')-\deg(\gamma_c')$ and $\ell=\deg(\gamma_a')+\deg(\gamma_c')$,
	this implies 	
	\color{black}
	all claimed properties of $\hat\lambda,\hat\Psi$, except the strict monotonicity of their phases. 
	That monotonicity  follows from the nondegeneracy condition
	\begin{align*}
		0 < |\det(\gamma'')|=|\det(\hat\lambda'\otimes \hat\Psi + \hat\lambda\otimes \hat\Psi')|
		=|\hat\lambda'| \: |\hat\Psi'|,
	\end{align*}
	where the
	\color{black}
	penultimate equality follows from \eqref{eq:lambdaPsi}, and the last equality from the fact that
 $|\hat{\lambda}|=|\hat{\Psi}|=1$.
 \color{black}
 This implies $|\hat\lambda'|,|\hat\Psi'|>0$ on $\R$.
\end{proof}

\begin{rem}\label{r:Psi}
	As a consequence of Lemma~\ref{l:lambdaPsi}, we can define the map
	\begin{align*}
		\Psi\in C^1(\mathbb S^1;\mathbb S^1),\quad &\Psi(e^{i\theta})=\hat\Psi(\theta)
		&&\text{if }k\text{ is even,}\\
		\Psi\in C^1(\mathbb S^1;\mathbb{RP}^1),\quad
		&\Psi(e^{i\theta})=\lbrace \pm \hat \Psi(\theta)\rbrace
		&&\text{if }k\text{ is odd}.
	\end{align*}
	This map has winding number $\deg(\Psi)=k/2$, and is a $C^1$ diffeomorphism if $|k|\in\lbrace 1,2\rbrace$. 
	\color{black}
	(Recall our convention \eqref{eq:degPsi} that the winding number of a loop in $\mathbb{RP}^1$ is a half-integer.)
	\color{black}
	Similarly we can define $\lambda$, $\mathbb S^1$-valued or $\mathbb{RP}^1$-valued, so that $\partial_{\theta}\Gamma = \lambda\otimes\Psi$. 
It then follows 
\color{black}
 that the image of $(\mathrm{cof}\,\gamma'(\theta))^T=(\partial_\theta\Gamma(e^{i\theta}))^T$
 \color{black}
  is spanned by $\Psi(e^{i\theta})$.
\end{rem}

\begin{rem}\label{r:Psi_bdry}
If $\gamma: [a,b]=I\to \R^{2\times 2}$ is a curve with boundary and satisfies the assumptions of Lemma~\ref{l:lambdaPsi}, then the same construction as in the proof of Lemma~\ref{l:lambdaPsi} gives $\hat\lambda, \hat\Psi \in C^1([a,b];\mathbb S^1)$ such that
	\begin{equation*}
		\mathrm{cof}\, \gamma'(\theta)=\partial_\theta\Gamma(e^{i\theta})
		=\hat\lambda(\theta)\otimes \hat\Psi(\theta) \qquad\forall \theta\in[a,b].
	\end{equation*}
	Then, as in Remark~\ref{r:Psi}, we can define $\Psi(e^{i\theta})=\hat\Psi(\theta), \lambda(e^{i\theta})=\hat\lambda(\theta) \in C^1(J;\mathbb S^1)$ for $J=\exp(iI)$. 
\end{rem}

\section{\color{black}A family of conservation laws in the nowhere elliptic case: proof of Proposition~\ref{p:conslaws}}\label{s:conslaws}

\color{black}
In this section we prove Proposition~\ref{p:conslaws}.
We let
 $J\subset\mathbb S^1$ be compact and connected,
 and $\Gamma\in C^2(J;\R^{2\times 2})$ satisfy $|\partial_\theta\Gamma|=1$, $\det(\partial_\theta\Gamma)=0$ and $|\det(\partial_\theta^2\Gamma)|>0$  on $J$.
We assume that  $m\colon \Omega\to J\subset\mathbb S^1$
has the regularity $B^{1/3}_{p,\infty,\mathrm{loc}}(\Omega;\mathbb S^1)$ for some $p>3$
and  satisfies \eqref{eq:divGamma0}, that is, the vector fields $\Gamma_1(m),\Gamma_2(m)$ are divergence-free.

Our goal is to show that $m$ solves the family of conservation laws \eqref{eq:conslaws}.
\color{black}
The starting point is the identity which follows from the chain rule,
\begin{align*}
\nabla\cdot\Phi(m)=\sum_{j=1,2}\alpha_\Phi^j(m)\nabla\cdot\Gamma_j(m)\qquad\text{if }m\colon\Omega\to J\subset \mathbb S^1\text{ is smooth.}
\end{align*}
In our case, $m$ is not smooth: in order to apply the chain rule, we use smooth approximating maps $m_\e\to m$.
The maps $m_\e$ will in general not be $\mathbb S^1$-valued, which introduces error terms in the above identity, 
and the proof of Proposition~\ref{p:conslaws} consists in estimating these error terms efficiently.
If the quantity $|h|^{-1/3}\|D^h m\|_{L^p}$ vanishes as $|h|\to 0$,
this follows directly from commutator estimates as in \cite{CET94,DLI15}.
Here we only know that this quantity is bounded \eqref{eq:besov}, 
and basic commutator estimates are not enough to conclude directly:
they only provide the information that entropy productions are in $L^{p/3}$, see \S~\ref{sss:Lp3}.
However they also provide identities involving weak limits of error terms, and we can take advantage of these identities and the structure of $\mathrm{ENT}_\Gamma$ to conclude.

\subsection{Commutator estimates}\label{ss:comm}

\color{black}
One of the main technical tools in the proof of Proposition~\ref{p:conslaws} is a family of commutator estimates.
Such estimates are well known, for the readers' convenience we state here the version that we will use.
\color{black}

\begin{lem}\label{l:commut}
Let $\Omega\subset\R^n$ be open and $w\in B^s_{p,\infty,\mathrm{loc}}\cap L^\infty(\Omega;\R^k)$ for some $s\in (0,1)$ and $p> 1$.
Let $\rho\in C_c^1(\R^n)$ with $\mathrm{supp}\,\rho\subset B_1$ and $\int\rho\, dx=1$ and denote convolution with $\rho_\e(x)=\e^{-n}\rho(x/\e)$ by a subscript $\e$.
Then for any $\alpha,\beta\geq 1$, any $C^2$ map $G\colon\R^k\to\R$, we have
\begin{align*}
\int_U
|G(w_\e)-G(w)_\e|^\alpha |\nabla w_\e|^\beta\, dx
\leq c \e^{s \min(p,2\alpha+\beta)-\beta}\qquad\forall U\subset\subset\Omega,
\end{align*}
for some constant $c>0$ depending on $\|G\|_{C^2}$, $\|w\|_\infty$,  $|w|_{B^{s}_{p,\infty}(U)}$, $\|\rho\|_{C^1}$, $\alpha$, $\beta$, $p$ and $U$.
\end{lem}
\begin{proof}[Proof of Lemma~\ref{l:commut}]
For any $U\subset\subset\Omega$ the regularity $B^s_{p,\infty,\mathrm{loc}}$ of $w$ amounts to
\begin{align*}
\int_U |D^h w|^p \leq |w|_{B^{s}_{p,\infty}(U)}^p |h|^{sp},
\end{align*}
for all $|h|\leq 1$.
Moreover, letting $R=\|w\|_\infty$, for any $\e>0$ we have (see e.g. \cite[Lemma~17]{LLP20})
\begin{align*}
|G(w_\e)-G(w)_\e|&
\leq c_0 \|G\|_{C^2(B_R)}\int_{B_1} |D^{\e y}w|^2\rho(y)\, dy,\\
|\nabla w_\e |
&\leq \frac{c_0}{\e} \int_{B_1} |D^{\e z}w|\, |\nabla \rho|(z)\, dz,
\end{align*}
for some absolute constant $c_0>0$, so applying Jensen's inequality and Fubini's theorem we deduce
\begin{align*}
&
\int_U
|G(w_\e)-G(w)_\e|^\alpha |\nabla w_\e|^\beta\, dx\\
&
\leq
\frac{c}{\e^\beta}\int_{B_1}\int_{B_1} \rho(y) |\nabla\rho|(z)\int_U |D^{\e y}w(x)|^{2\alpha}|D^{\e z}w(x)|^\beta \, dx\, dy dz,
\end{align*}
for some constant $c$ depending on $\|G\|_{C^2(B_R)}$ and $\|\rho\|_{C^1}$.
For any $y,z\in B_1$ we have, by H\"older's inequality with exponents $q=(2\alpha+\beta)/(2\alpha)$ and $q'=(2\alpha+\beta)/\beta$,
\begin{align*}
\int_U |D^{\e y}w|^{2\alpha}|D^{\e z}w|^\beta \, dx
&
\leq \left(\int_U |D^{\e y}w|^{2\alpha+\beta}\, dx\right)^{\frac 1q}
\left(\int_U |D^{\e z}w|^{2\alpha+\beta}\, dx\right)^{1-\frac 1q}\\
&\leq c |w|_{B^{s}_{p,\infty}(U)}^{\min (p,2\alpha+\beta)}\e^{s\min (p,2\alpha+\beta)}.
\end{align*}
In the last line we used the Besov regularity assumption, and the constant
$c$ depends on $\|w\|_\infty$, $|U|, \alpha, \beta$ and $p$. Plugging this back into the previous estimate concludes the proof.
\end{proof}

\subsection{The case without boundary}

We first prove Proposition~\ref{p:conslaws} in the case of a closed curve, $J=\mathbb S^1$.
The case with boundary will be treated in \S~\ref{ss:conslaws_bdry}.

\subsubsection{First step: entropy productions are in $L^{p/3}_{\mathrm{loc}}$}\label{sss:Lp3}

In this first step we 
 fix $\Phi\in \mathrm{ENT}_\Gamma $ and prove, using the regularity $m\in B^{1/3}_{p,\infty}$ and commutator estimates,
 that $\nabla\cdot\Phi(m)\in L^{p/3}_{\mathrm{loc}}(\Omega)$.
In fact we do not use all the assumptions on $\Gamma$ in Proposition~\ref{p:conslaws} and prove a more precise statement.

\begin{prop}\label{p:Lp3}
Let $\Gamma\in C^2(\mathbb S^1;\R^{2\times 2})$ and 
$m\in B^{\frac 13}_{p,\infty,\mathrm{loc}}(\Omega;\mathbb S^1)$ for some $p>3$, such that 
\begin{align*}
\nabla\cdot \Gamma_1(m)=\nabla\cdot\Gamma_2(m)=0\qquad
\text{ in }\mathcal D'(\Omega).
\end{align*}
Then there exist $f_1,f_2\in L^{p/3}_{\mathrm{loc}}(\Omega;\R^2)$ depending on $\Gamma$ and $m$ such that, for any entropy $\Phi\in \mathrm{ENT}_\Gamma$ \eqref{eq:ENT},
we have
\begin{align}\label{eq:dvPhif1f2}
\nabla\cdot\Phi(m)
=
\sum_{j=1}^2\partial_\theta\alpha_\Phi^j(m)\, im\cdot f_j \in L^{p/3}_{\mathrm{loc}}(\Omega),
\end{align}
where $\alpha_\Phi^1,\alpha_\Phi^2\in C^1(\mathbb S^1;\R)$ are as in \eqref{eq:ENT}.
\end{prop}

\begin{proof}[Proof of Proposition~\ref{p:Lp3}]
Direct calculations using polar coordinates $x=re^{i\theta}$ in $\R^2$ show that, for any $C^1$ vector fields $w\colon\Omega\to\R^2$ and $F\colon \R^2\to\R^2$ such that $\nabla F(0)=0$, we have
\begin{align}\label{eq:divFw}
\nabla\cdot F(w)
=
\left(\left(\partial_r F(w)\cdot\nabla \right)w\right)\cdot\frac{w}{|w|}
+
\left(\left(\partial_\theta F(w)\cdot\nabla \right)w\right)\cdot\frac{iw}{|w|^2}.
\end{align}
We apply \eqref{eq:divFw} to
vector fields $F$  which are `flat radial' extensions of $\Phi$ and $\Gamma_j$ ($j=1,2$) given in polar coordinates by
\begin{align*}
\widehat\Phi(re^{i\theta})=\eta(r)\Phi(e^{i\theta}),\qquad \widehat\Gamma_j(re^{i\theta})=\eta(r)\Gamma_j(e^{i\theta}),
\end{align*}
for some smooth $\eta\in C_c^\infty((0,\infty);[0,1])$ such that $\eta (r)\equiv 1$ for $1/2\leq r\leq 2$.
Recall from the definition \eqref{eq:ENT} of $\mathrm{ENT}_\Gamma$ that 
\begin{align*}
\partial_\theta\Phi =\sum_{j=1}^2 \alpha_\Phi^j \partial_\theta\Gamma_j,\qquad
\alpha_\Phi^j\in C^1(\mathbb S^1;\R).
\end{align*}
From \eqref{eq:divFw} applied to $\widehat\Phi$ and $\widehat\Gamma_j$ we therefore obtain,
 for any smooth $w\colon\Omega\to\R^2$,  
\begin{align*}
\nabla\cdot \widehat \Phi(w)
&
=
\eta'(|w|)\left(\left(\Phi\left(\frac{w}{|w|}\right)\cdot\nabla \right)w\right)\cdot\frac{w}{|w|}
\\
&\quad
+
\eta(|w|)\sum_{j=1}^2\alpha_\Phi^j\left(\frac{w}{|w|}\right)\left(\left( \partial_\theta\Gamma_j\left(\frac{w}{|w|}\right)\cdot\nabla \right)w\right)\cdot\frac{iw}{|w|^2},
\\
\nabla\cdot \widehat \Gamma_j(w)
&
=
\eta'(|w|)\left(\left(\Gamma_j\left(\frac{w}{|w|}\right)\cdot\nabla \right)w\right)\cdot\frac{w}{|w|}
\\
&\quad
+
\eta(|w|)\left(\left(\partial_\theta\Gamma_j\left(\frac{w}{|w|}\right)\cdot\nabla \right)w\right)\cdot\frac{iw}{|w|^2},
\end{align*}
which implies
\begin{align*}
\nabla\cdot \widehat \Phi(w)
-
\sum_{j=1,2}\alpha_\Phi^j\left(\frac{w}{|w|}\right)\nabla\cdot\widehat\Gamma_j(w)
&
=
\eta'(|w|)\left(\left(H\left(\frac{w}{|w|}\right)\cdot\nabla \right)w\right)\cdot\frac{w}{|w|},
\end{align*}
where $H=\Phi-\sum\alpha_\Phi^j\Gamma_j\in C^1(\mathbb S^1;\R^2)$.
Now we apply this to $w=m_\e$, where
\begin{align*}
m_\e =m *\rho_\e,\quad \rho_\e(x)=\e^{-2}\rho(x/\e),
\end{align*}
for some smooth kernel $\rho$ with support in $B_1$, 
and deduce
\begin{align}\label{eq:divPhiR1R2}
\nabla\cdot\widehat\Phi(m_\e)
&=R^1_\e +R^2_\e,\\
R^1_\e
&=
\sum_{j=1,2}\alpha_\Phi^j\left(\frac{m_\e}{|m_\e|}\right)\nabla\cdot
\left( \widehat\Gamma_j(m_\e)-\widehat\Gamma_j(m)_\e\right),
\nonumber
\\
R^2_\e
&=\eta'(|m_\e|)\left(\left(H\left(\frac{m_\e}{|m_\e|}\right)\cdot\nabla \right)m_\e\right)\cdot\frac{m_\e}{|m_\e|}.
\nonumber
\end{align}
To obtain the expression of $R^1_\e$ we used the fact that $\nabla\cdot\widehat\Gamma_j(m)=\nabla\cdot\Gamma_j(m)=0$.
Because $\eta'(r)\equiv 0$ in a neighborhood of $r=1$ we have
\begin{align*}
|R^2_\e|\leq C (1-|m_\e|^2)^2|D m_\e|,
\end{align*}
for some $C>0$ depending on $\eta$ and $H$, and 
noting that $1=|m|^2=(|m|^2)_\e$ we deduce 
from  commutator estimates  (Lemma~\ref{l:commut} applied to $G=|\cdot|^2$, $w=m$, $\alpha=2$, $\beta=1$, $s=1/3$, $p>3$) that
\begin{align}\label{eq:R2L1loc}
R^2_\e \longrightarrow 0\qquad\text{ in }L^1_{\mathrm{loc}}(\Omega).
\end{align}
Using that $\eta(r)\equiv 1$ in a neighborhood of $r=1$,
the same argument shows that, for the term $R^1_\e$ in \eqref{eq:divPhiR1R2}, we have
\begin{align}\label{eq:R1cutoffL1loc}
&(1-\eta(|m_\e|))R^1_\e 
\nonumber
\\
&
=(1-\eta(|m_\e|)) \sum_{j=1,2}\alpha_\Phi^j\left(\frac{m_\e}{|m_\e|}\right)\nabla\cdot\widehat\Gamma_j(m_\e) 
\to 0\qquad\text{ in }L^1_{\mathrm{loc}}(\Omega).
\end{align}
It remains to estimate $\eta(|m_\e|)R^1_\e$, which, 
defining
$\hat\alpha_\Phi^j(re^{i\theta})=\eta(r)\alpha_\Phi^j(e^{i\theta})\in C_c^1(\R^2;\R)$, we rewrite 
as
\begin{align*}
\eta(|m_\e|)R^1_\e
&
=\nabla\cdot
\left( \sum_{j=1,2}
\hat\alpha_\Phi^j(m_\e)
\lt(\widehat\Gamma_j(m_\e)-\widehat\Gamma_j(m)_\e\rt)
\right) \nn
\\
& 
\quad
-\sum_{j=1,2} \lt(\nabla\hat\alpha_\Phi^j(m_\e) D m_\e\rt) \cdot
\left( \widehat\Gamma_j(m_\e)-\widehat\Gamma_j(m)_\e\right).
\end{align*}
The first line in the right-hand side is the divergence of a sequence which tends to 0 in $L^1_{\mathrm{loc}}(\Omega)$. 
Each term in the second line is bounded in $L_{\mathrm{loc}}^{p/3}$, independently of $\e$, by commutator estimates
(Lemma~\ref{l:commut} applied to $G=\widehat\Gamma_j$, $w=m$, $s=1/3$, $\alpha=\beta=p/3$). 
More precisely, there exist 
 $f_1,f_2\in L^{p/3}_{\mathrm{loc}}(\Omega;\R^2)$
and a sequence $\e_k\to 0$ such that
\begin{align*}
 D m_\e
\left(\widehat\Gamma_j(m)_\e 
-\widehat\Gamma_j(m_\e)
\right)\rightharpoonup f_j\qquad\text{ in }L^{\frac p3}_{\mathrm{loc}}(\Omega),
\end{align*}
and therefore
\begin{align*}
\eta(|m_{\e_k}|)R^1_{\e_k}
\to \sum_{j=1}^2 \nabla \hat\alpha_\Phi^j(m)\cdot f_j
=
\sum_{j=1}^2\partial_\theta\alpha_\Phi^j(m)\, im\cdot f_j\qquad\text{in }\mathcal D'(\Omega).
\end{align*}
Plugging this and \eqref{eq:R2L1loc}-\eqref{eq:R1cutoffL1loc} into \eqref{eq:divPhiR1R2}, we deduce
\begin{align*}
\nabla\cdot\widehat\Phi(m_{\e_k})\to
 \sum_{j=1}^2\partial_\theta\alpha_\Phi^j(m)\, im\cdot f_j
\qquad\text{in }\mathcal D'(\Omega).
\end{align*}
Since $\widehat\Phi(m_\e)\to\Phi(m)$ in $L^1_{\mathrm{loc}}(\Omega)$, this implies
\eqref{eq:dvPhif1f2}.
\end{proof}

\subsubsection{Conclusion: entropy productions vanish}\label{sss:choice_ent_NE}

Here we show that all entropy productions vanish.
If $\Phi\in \mathrm{ENT}_\Gamma$ is such that there exist $\beta_\Phi^1,\beta_\Phi^2\in C^1(\mathbb S^1;\R)$ for which
\begin{align}\label{eq:partialPhi}
\partial_\theta\Phi =\beta_\Phi^1\partial_\theta\Gamma_1 =\beta_\Phi^2\partial_\theta\Gamma_2,
\end{align}
then applying \eqref{eq:dvPhif1f2} we obtain
\begin{align*}
\left(
\begin{array}{c}
\partial_\theta\beta_\Phi^1(m)
\\
-\partial_\theta\beta_\Phi^2(m)
\end{array}
\right)
\cdot
\left(
\begin{array}{c}
im\cdot f_1
\\
im\cdot f_2\end{array}
\right)
=0.
\end{align*}
So if we can find two such maps $\Phi,\overline\Phi$ such that
\begin{align*}
\det
\left(
\begin{array}{cc}
\partial_\theta\beta_\Phi^1(m)
&
\partial_\theta\beta_{\overline \Phi}^1(m)
\\
-\partial_\theta\beta_\Phi^2(m)
&
-\partial_\theta\beta_{\overline\Phi}^2(m)
\end{array}
\right)
\neq 0\quad\text{a.e.},
\end{align*}
then we deduce that $im\cdot f_1=im\cdot f_2=0$ a.e. and the conclusion \eqref{eq:conslaws} follows from \eqref{eq:dvPhif1f2}.

Note that here all assumptions of Lemma~\ref{l:lambdaPsi} are satisfied, so we have (recalling Remark~\ref{r:Psi})
\begin{align*}
	\partial_{\theta}\Gamma = \lambda\otimes\Psi, 
\end{align*}
for $\lambda,\Psi\in C^1(\mathbb S^1;\mathbb{RP}^1)$. Assume for example that there are entropies $\Phi,\overline\Phi$ such that
\begin{align}\label{eq:PhiPhibar}
\partial_\theta\Phi =\lambda_1^2\lambda_2\Psi,
\qquad
\partial_\theta\overline\Phi =\lambda_2^2\lambda_1\Psi.
\end{align}
The right-hand sides belong to 
$C^1(\mathbb S^1;\R^2)$ 
because 
$\lambda\in C^1(\mathbb S^1;\mathbb{RP}^1)$ and $\lambda\otimes\Psi\in C^1(\mathbb S^1;\R^{2\times 2})$.
These maps $\Phi,\overline\Phi$ would then satisfy \eqref{eq:partialPhi}, with
\begin{align*}
\beta_\Phi^1 =\lambda_1\lambda_2,\quad\beta_\Phi^2=\lambda_1^2,
\quad
\beta_{\overline\Phi}^1
=\lambda_2^2,
\quad
\beta_{\overline\Phi}^2=\lambda_1\lambda_2,
\end{align*}
and the above determinant is equal to
\begin{align*}
\det
\left(
\begin{array}{cc}
\partial_\theta\beta_\Phi^1
&
\partial_\theta\beta_{\overline \Phi}^1
\\
-\partial_\theta\beta_\Phi^2
&
-\partial_\theta\beta_{\overline\Phi}^2
\end{array}
\right)
&
=-(\lambda_2\partial_\theta\lambda_1
+\lambda_1\partial_\theta\lambda_2)^2
+4\lambda_1\lambda_2\partial_\theta\lambda_1\partial_\theta\lambda_2
\\
&
=-(\lambda_2\partial_\theta\lambda_1
-\lambda_1\partial_\theta\lambda_2)^2
=
-(i\lambda\cdot\partial_\theta\lambda)^2
\\
& =-|\partial_\theta\lambda|^2 <0.
\end{align*}
The last equality follows from
the fact that
$\lt|\lambda\rt|=1$,
 and the last inequality follows from the monotonicity of its phase as shown in Lemma~\ref{l:lambdaPsi}.

In general there might however not exist entropies 
whose derivatives are exactly as in \eqref{eq:PhiPhibar},
because the right-hand sides in \eqref{eq:PhiPhibar} may not have zero average.
But for any fixed $x\in\Omega$, we
 can modify
the maps $\lambda_1^2\lambda_2\Psi,\lambda_2^2\lambda_1\Psi\in C^1(\mathbb S^1;\R^2)$ on a subset of $\mathbb S^1$ which does not contain a small interval around $m(x)$, so that the modified maps have zero averages and are derivatives of entropies $\Phi,\overline\Phi$ as above. 
\color{black}
More precisely, we claim that there exist entropies $\Phi,\overline\Phi$ such that
\begin{align}\label{eq:modif_claim}
\partial_\theta\Phi =\lambda_1^2\lambda_2\Psi
\text{ and }
\partial_\theta\overline\Phi =\lambda_2^2\lambda_1\Psi
\text{ in a neighborhood of }m(x)\,.
\end{align}
\color{black}
At the point $x$, 
the determinant has the same value as before, so we deduce that $im(x)\cdot f_1(x)=im(x)\cdot f_2(x)=0$. 
We can do this for any fixed $x$,
so
the right-hand side of \eqref{eq:dvPhif1f2} is zero a.e., 
and
 this concludes the proof that all entropy productions vanish,
 \color{black} 
once we will have proved the claimed existence of entropies $\Phi,\overline\Phi$ satisfying \eqref{eq:modif_claim}.
%
%
To that end, we state here a technical lemma that will also be useful
in other places.\color{black}

\begin{lem}\label{l:modif} 
Let $\Gamma\in C^2(\mathbb S^1;\R^{2\times 2})$ be such that  
$\det(\partial_\theta^2\Gamma)$ is not identically zero.
Let $\Gamma_1,\Gamma_2$ denote the two rows of the matrix $\Gamma$.
 For any $\alpha_1,\alpha_2\in C^1(\mathbb S^1;\R)$, there exist $ \tilde\alpha_1,  \tilde\alpha_2\in C^1(\mathbb S^1;\R)$ such that
\begin{align*}
&
\int_{\mathbb S^1}\left(\sum_{j=1}^2\tilde\alpha_j \partial_\theta\Gamma_j \right)\, d\theta
=0
\\
\text{and }
&
\sum_{j=1}^2\sup_{\mathbb S^1}|\alpha_j-\tilde\alpha_j|
 \leq C \left|
\int_{\mathbb S^1}\left(\sum_{j=1}^2\alpha_j \partial_\theta\Gamma_j \right)\, d\theta\right|,
\end{align*}
for some constant $C$ depending only on $\Gamma$.
{\color{black}
Moreover, we can choose the functions $\tilde\alpha_j$ such that $\tilde\alpha_j=\alpha_j$ away from any neighborhood of any $\lbrace z_1,z_2\rbrace\subset\mathbb S^1$ such that $\det(\partial_\theta\Gamma_1(z_1),\partial_\theta\Gamma_2(z_2))\neq 0$.}
\end{lem}

\begin{proof}[Proof of Lemma~\ref{l:modif}]
We denote by $v\in\R^2$ the integral
\begin{align*}
v=
\int_{\mathbb S^1}\left(\sum_{j=1}^2\alpha_j \partial_\theta\Gamma_j \right)\, d\theta,
\end{align*}
and prove the existence of
$\beta_1,\beta_2\in C^1(\mathbb S^1;\R)$ such that
\begin{align}\label{eq:condbetaj}
&
\int_{\mathbb S^1}\left(\sum_{j=1}^2\beta_j \partial_\theta\Gamma_j \right)\, d\theta
=
v
\qquad\text{and}\qquad
\sum_{j=1}^2\sup_{\mathbb S^1}
|\beta_j|
 \leq C  |v|.
\end{align}
Then it suffices to define $\tilde\alpha_j=\alpha_j-\beta_j$.

We look for $\beta_j$ in the form $\beta_j(e^{i\theta})= b_j f_j(e^{i\theta})$, for some $b\in\R^2$ and $f_1,f_2\in C^1(\mathbb S^1;\R)$, then the first identity in \eqref{eq:condbetaj} amounts to
\begin{align*}
Ab=v,\quad\text{where }
A=\left(
\begin{array}{c|c}
\int_{\mathbb S^1}f_1\partial_\theta\Gamma_1
\, d\theta
&
\int_{\mathbb S^1} f_2\partial_\theta\Gamma_2\, d\theta
\end{array}
\right)\in\mathbb R^{2\times 2}.
\end{align*}
So it suffices to choose $f_1,f_2$ such that $\det(A)\neq 0$,
then we simply set $b=A^{-1}v$, and the second estimate in \eqref{eq:condbetaj} follows from $|b|\leq \| A^{-1}\| |v|$.

To find $f_1,f_2$ such that $\det(A)\neq 0$,
we can pick for instance
 $z_1,z_2\in\mathbb S^1$ such that 
$\det(\partial_\theta\Gamma_1(z_1),\partial_\theta\Gamma_2(z_2))\neq 0$ 
and choose, 
for $j=1,2$,
 $f_j$ approximating the Dirac mass at $z_j$.
The existence of such $z_1,z_2$
 follows from the fact that $\det(\partial^2_\theta\Gamma)$ is not identically zero:
if $\det(\partial_\theta\Gamma_1(z_1),\partial_\theta\Gamma_2(z_2))=0$ for all $z_1,z_2\in\mathbb S^1$ then
 $\partial_\theta\Gamma_1(\mathbb S^1)$ and $\partial_\theta\Gamma_2(\mathbb S^1)$ are contained in a single line,
which implies that $\partial_\theta \Gamma  v_0=0$  is identically zero for some fixed $v_0\in\mathbb S^1$, and therefore $\det(\partial^2_\theta\Gamma)$ is identically zero. 
\end{proof}

%

\begin{proof}[\color{black}Proof of Proposition~\ref{p:conslaws} (case without boundary) completed]
 {\color{black}
 We can now prove the existence of entropies $\Phi,\overline\Phi$ satisfying \eqref{eq:modif_claim}, thus concluding the
 proof of Proposition~\ref{p:conslaws} in the case without boundary.}
 Recall from Lemma~\ref{l:lambdaPsi} that $\lambda,\Psi$ 
have uniformly monotone phases. 
This implies that  $\lambda_1^{-1}(\lbrace 0\rbrace)$, $\lambda_2^{-1}(\lbrace 0\rbrace)$  and
$\Psi^{-1}(\lbrace \pm \Psi(z)\rbrace)$ 
are finite subsets of $\mathbb S^1$ for any $z\in\mathbb S^1$.
As a consequence, we can find
$z_1=e^{i\theta_1},z_2=e^{i\theta_2}\in\mathbb S^1\setminus \lbrace m(x)\rbrace$ such that 
$\det(\partial_{\theta}\Gamma_1(z_1),\partial_{\theta}\Gamma_2(z_2))\neq 0$. Denote $\alpha_1=\lambda_1\lambda_2$, $\alpha_2=0$, so that $\sum_{j=1}^2\alpha_j \partial_\theta\Gamma_j=\lambda_1^2\lambda_2\Psi$. 
Then Lemma~\ref{l:modif} gives $\ti\alpha_1, \ti\alpha_2\in C^1(\mathbb S^1;\R)$ such that $\int_{\mathbb S^1}\left(\sum_{j=1}^2\tilde\alpha_j \partial_\theta\Gamma_j \right)\, d\theta
=0$ and $\sum_{j=1}^2\tilde\alpha_j \partial_\theta\Gamma_j =\sum_{j=1}^2\alpha_j \partial_\theta\Gamma_j $ away from a sufficiently small neighborhood of $\{z_1, z_2\}$. 
In particular, we have $\sum_{j=1}^2\tilde\alpha_j \partial_\theta\Gamma_j=\lambda_1^2\lambda_2\Psi$ in a neighborhood of $m(x)$. Therefore we can define $\Phi(e^{i\theta})=\int_0^\theta \lt(\sum_{j=1}^2\tilde\alpha_j \partial_t\Gamma_j\rt) dt$,
which satisfies the first equation of \eqref{eq:modif_claim} on an open interval around $m(x)$.
An entropy $\overline\Phi$
satisfying the second equation of \eqref{eq:modif_claim} on an open interval around $m(x)$ is obtained by the same argument. 
\end{proof}

\subsection{The case with boundary}\label{ss:conslaws_bdry}

We now prove Proposition~\ref{p:conslaws} in the case with boundary, that is, 
 $J\varsubsetneq \mathbb S^1$ is  compact and connected.
We start by showing that there exist
 $f_1,f_2\in L^{p/3}_{\mathrm{loc}}(\Omega)$ such that
\begin{align}\label{eq:divPhiLp3_bdry}
\nabla\cdot
\Phi(m)=\sum_{j=1}^2\partial_\theta\alpha_{\Phi}^j(m)\, im\cdot f_j,
\qquad\forall \Phi\in \mathrm{ENT}_{\Gamma}.
\end{align}
To prove \eqref{eq:divPhiLp3_bdry} we extend $\Gamma$ to a map $\widetilde\Gamma\in C^2(\mathbb S^1;\R^{2\times 2})$.
Then the map $m$ satisfies $\widetilde\Gamma(m)=\Gamma(m)$, hence $\nabla\cdot\widetilde\Gamma(m)=0$.
Applying Proposition~\ref{p:Lp3} to $\widetilde\Gamma$ and $m$, we deduce that
\begin{align}\label{eq:divtildePhiLp3}
\nabla\cdot\widetilde\Phi(m)=\sum_{j=1}^2\partial_\theta\alpha_{\widetilde\Phi}^j(m)\, im\cdot f_j,
\qquad\forall \widetilde\Phi\in \mathrm{ENT}_{\widetilde \Gamma},
\end{align}
for some $f_1,f_2\in L^{p/3}_{\mathrm{loc}}(\Omega)$.
Now let $\Phi\in \mathrm{ENT}_\Gamma$, that is,
\begin{align*}
\partial_\theta \Phi =\sum_{j=1}^2 \alpha_\Phi^j\partial_\theta\Gamma_j\qquad\text{on }J,
\end{align*}
for some $\alpha_\Phi^j\in C^1(J;\R)$. 
There exist extensions $\tilde\alpha_\Phi^j\in C^1(\mathbb S^1;\R)$ of $\alpha_\Phi^j$ such that
\begin{align}\label{eq:tildealphaj}
\int_{\mathbb S^1} \left(\sum_{j=1}^2 \tilde\alpha_\Phi^j\partial_\theta\widetilde\Gamma_j\right)\, d\theta
=0.
\end{align}
Indeed, 
one can choose any extensions $\tilde\alpha_\Phi^j$ and
modify them in a neighborhood of a point $z_0\in \mathbb S^1\setminus J$ at which $\det(\partial_\theta\widetilde\Gamma)\neq 0$ in order to ensure \eqref{eq:tildealphaj}, {\color{black}see Lemma~\ref{l:modif}.}
If such a point $z_0$ does not exist in the first place, one can simply modify $\widetilde\Gamma$ away from $J$ by adding a term of the form $\eta I$ 
for $I\in\R^{2\times 2}$ the identity matrix and
 for some
non-zero $\eta\in C_c^2(\mathbb S^1\setminus J;\R)$ 
with $\|\eta'\|_{\infty}$ sufficiently large.

Thanks to \eqref{eq:tildealphaj} one can then define $\widetilde\Phi\in \mathrm{ENT}_{\widetilde\Gamma}$ such that $\widetilde \Phi=\Phi$ on $J$, by setting
\begin{align*}
\widetilde\Phi(e^{i\theta})=\Phi(e^{i\theta_0}) +\int_{\theta_0}^\theta 
\left(\sum_{j=1}^2 \tilde\alpha_\Phi^j\partial_t\widetilde\Gamma_j\right)dt,
\end{align*}
for some $e^{i\theta_0}\in J$.
Applying \eqref{eq:divtildePhiLp3} to this $\widetilde\Phi$ and using that $m$ takes values into $J$ we deduce \eqref{eq:divPhiLp3_bdry}.

Now we can use \eqref{eq:divPhiLp3_bdry} exactly as in \S~\ref{sss:choice_ent_NE} and conclude that $f_1=f_2=0$. In fact it is even easier because entropies do not have to be periodic, so we can directly find entropies $\Phi,\overline\Phi\in \mathrm{ENT}_\Gamma$ such that
\begin{align*}
\partial_\theta\Phi =\lambda_1^2\lambda_2\Psi,
\quad
\partial_\theta\overline\Phi
=\lambda_2^2\lambda_1\Psi,
\end{align*}
where $\lambda,\Psi\in C^1(J;\mathbb S^1)$ are such that $\partial_\theta\Gamma=\lambda\otimes\Psi$ on $J$ (see Remark~\ref{r:Psi_bdry}).
\qed

\section{The nowhere elliptic case: proofs of Theorems~\ref{t:NE} and \ref{t:NE_bdry}}\label{s:proofNE}

\color{black}
Proposition~\ref{p:conslaws} is the main ingredient in the 
proofs of Theorems~\ref{t:NE} and \ref{t:NE_bdry}
 and has been proved in \S~\ref{s:conslaws}.
Here we show how Theorems~\ref{t:NE} and \ref{t:NE_bdry} follow from Proposition~\ref{p:conslaws}.
\color{black}

First note that thanks to
Lemma~\ref{l:besov}, under the assumptions of Theorem~\ref{t:reg_gen}, in particular the quartic estimate \eqref{eq:det4},  the map $m$ given by $Du=\mathrm{cof}\, \Gamma(m)$ \eqref{eq:cofDuGammam} has the fractional regularity $B^{1/3}_{4,\infty,\mathrm{loc}}$ and we can apply Proposition~\ref{p:conslaws}:  the map $m$ solves
the family of conservation laws
\eqref{eq:conslaws}. 
In this section we use  this family of conservation laws to obtain the conclusions of Theorems~\ref{t:NE} 
and \ref{t:NE_bdry}. 

This is done in several steps and the map $\Psi$ appearing in Theorems~\ref{t:NE} and \ref{t:NE_bdry} plays a crucial role. 
\color{black} Recall that we obtained an explicit expression of $\Psi$ in Lemma~\ref{l:lambdaPsi}.
\color{black}
First we check in \S~\ref{ss:zerostates} that the map
$\tilde m=i\Psi(m)$  is a zero-energy state of the eikonal equation, in the sense of \cite{JOP02} if $\Psi$ is $\mathbb S^1$-valued, and in the sense of \cite{GMPS23} if $\Psi$ is $\mathbb{RP}^1$-valued. 
\color{black}
(This idea has already appeared in 
\cite[Theorem~4]{LLP22}, see Remark~\ref{r:geneik}.) 
\color{black}
In particular, this gives parts $(b)$ and $(c)$ of Theorem~\ref{t:NE}, in which cases $\Psi$ is a $C^1$ diffeomorphism and $m$ can have rigid singularities. 
Then in \S~\ref{ss:high_deg}, we use geometric arguments to show that the singular set of $m$ is empty in the case where the curve is closed and $\Psi$ has a high winding number. This gives part $(a)$ of Theorem~\ref{t:NE}. The same geometric arguments allow to show that the singular set of $m$ is empty if the curve has boundary, finally giving Theorem~\ref{t:NE_bdry} in \S~\ref{ss:NE_bdry}.

\subsection{Relation to zero-energy states of Aviles-Giga functionals}
\label{ss:zerostates}

First recall the notations from the introduction. Specifically, we
	assume without loss of generality that $\Pi$ has length 
	at most $2\pi$, and $\gamma\colon  I \to\Pi$ is an arc-length parametrization of class $C^2$,
	where $I$ is either a segment $[a,b]\subset\R$ ($a<b<a+2\pi$) in the case with boundary, or $I=\R/2\pi\Z$ in the case without boundary. Further, we set $J=\exp(iI)\subset\mathbb S^1$. 
	
	Here, 
	for $m:\Omega\to J$ solving the family of conservation laws \eqref{eq:conslaws},
	 we check that the map
$\tilde m=i\Psi(m)$  is a zero-energy state of the eikonal equation, in the sense of \cite{JOP02} if $\Psi$ is $\mathbb S^1$-valued, and in the sense of \cite{GMPS23} if $\Psi$ is $\mathbb{RP}^1$-valued.
This is due to the fact that, if $\widetilde\Phi \colon\mathbb S^1\to\R^2$ is an entropy of the eikonal equation in the sense of \cite{DKMO01}, that is,
\begin{align*}
	\partial_\theta\widetilde\Phi(e^{i\theta})\cdot e^{i\theta} =0\qquad\forall\theta\in\R,
\end{align*}
then (restricting to even maps $\widetilde\Phi$ in the case where $\Psi$ is $\mathbb{RP}^1$-valued) the map $\widetilde\Phi\circ (i\Psi)$ is an entropy in the sense of the present paper, and therefore $\nabla\cdot\widetilde\Phi(\tilde m)=0$, which is enough to apply the regularity and rigidity results of \cite{JOP02,GMPS23}  to the map $\tilde m$.

Since $\Psi$, and therefore $\widetilde\Phi\circ (i\Psi)$, is only $C^1$,
we must first extend \eqref{eq:conslaws} to entropies $\Phi$ which are only $C^1$.

\begin{lem}\label{l:conslawsC1}
Assume that $\Gamma\in C^2(J;\R^{2\times 2})$ satisfies $\det(\partial_{\theta}^2\Gamma)\ne 0$ on $J$. If $m\colon\Omega\to J$ satisfies
	\begin{align*}
		\nabla\cdot\Phi(m)=0\qquad\text{in }\mathcal D'(\Omega),\qquad\forall\Phi\in \mathrm{ENT}_\Gamma,
	\end{align*}
	then this identity is valid for all maps $\Phi$ in the larger class
	\begin{align*}
		\mathrm{ENT}_\Gamma^1 =\Big\lbrace
		\Phi\in C^1(J;\R^2)
		\colon
		&
		\exists \alpha_\Phi^1,\alpha_\Phi^2\in C^0(J;\R)\text{ s.t. }
		\nonumber
		\\
		&
		\partial_\theta\Phi =\alpha_\Phi^1 \partial_\theta\Gamma_1 +\alpha_\Phi^2 \partial_\theta\Gamma_2
		\Big\rbrace.
	\end{align*}
\end{lem}
\begin{proof}[Proof of Lemma~\ref{l:conslawsC1}]
	This follows directly from the fact that  $\mathrm{ENT}_\Gamma$ is dense in $\mathrm{ENT}_\Gamma^1$ in the $C^1$ topology.
	Let indeed $\Phi\in \mathrm{ENT}^{1}_\Gamma$, and $\alpha^j :=\alpha_\Phi^j\in C^0(J;\R)$.
	There exist $\alpha^j_k\in C^1(J;\R)$ such that $\alpha^j_k\to \alpha^j$ uniformly. 

If $J=\lbrace e^{i\theta}\rbrace_{\theta\in [a,b]}\varsubsetneq \mathbb S^1$, the formula
	\begin{align*}
		\Phi_k(e^{i\theta})=\Phi(e^{ia})+\int_a^\theta 
		\left(
		\sum_{j=1}^2 \alpha^j_k\partial_{t}\Gamma_j
		\right)
		(e^{it})
		\, dt\qquad\forall\theta\in[a,b],
	\end{align*}
	defines an entropy $\Phi_k\in \mathrm{ENT}_\Gamma$, and $\Phi_k\to\Phi$ in the $C^1$ topology since $\Phi_k(e^{ia})=\Phi(e^{ia})$ and $\partial_\theta\Phi_k\to\partial_\theta\Phi$ uniformly.

If $J=\mathbb S^1$, the average of 
	$\sum_{j=1}^2\alpha^j_k\partial_\theta\Gamma_j$ on $\mathbb S^1$ tends to the average of 
	$\sum_{j=1}^2\alpha^j\partial_\theta\Gamma_j
	=\partial_\theta\Phi$ on $\mathbb S^1$, which is equal to zero. Therefore, applying Lemma~\ref{l:modif}, we obtain $\tilde\alpha^j_k\in C^1(\mathbb S^1;\R)$ such that $\tilde\alpha^j_k\to \alpha^j$ uniformly and $\sum_{j=1}^2\tilde\alpha^j_k\partial_\theta\Gamma_j$ has zero average on $\mathbb S^1$.
	Then the formula
	\begin{align*}
		\Phi_k(e^{i\theta})=\Phi(1)+\int_0^\theta 
		\left(
		\sum_{j=1}^2 \tilde\alpha^j_k\partial_{t}\Gamma_j
		\right)
		(e^{it})
		\, dt
	\end{align*}
	defines an entropy $\Phi_k\in \mathrm{ENT}_\Gamma$, and $\Phi_k\to\Phi$ in the $C^1$ topology since $\Phi_k(1)=\Phi(1)$ and $\partial_\theta\Phi_k\to\partial_\theta\Phi$ uniformly.
\end{proof}

Now that entropies are allowed to be $C^1$, we can check that entropies of the eikonal equation provide entropies in our setting.

\begin{lem}\label{l:entropiesAG}
Assume that $\Gamma\in C^2(J;\R^{2\times 2})$ satisfies $|\partial_\theta\Gamma|=1$, $\det(\partial_\theta\Gamma)=0$ and $|\det(\partial_\theta^2\Gamma)|>0$  on $J$.
 Let $\widetilde\Phi\in C^1(\mathbb S^1;\R^2)$ be such that 
	$\partial_\theta\widetilde\Phi(e^{i\theta})\cdot e^{i\theta}=0$ for all $\theta\in\R$,
	and in the case where 
	$\Psi$ is $\mathbb{RP}^1$-valued assume in addition that $\widetilde\Phi$ is even.
	Then $\Phi=\widetilde\Phi\circ (i\Psi)\in \mathrm{ENT}_{\Gamma}^1$.
\end{lem}
\begin{proof}[Proof of Lemma~\ref{l:entropiesAG}]
	Let 
	$\mu(e^{i\theta})=\partial_\theta\widetilde\Phi(e^{i\theta})\cdot ie^{i\theta} $,
	so that $\mu\in C^0(\mathbb S^1;\R)$ and $\partial_\theta\widetilde\Phi(e^{i\theta})=\mu\lt(e^{i\theta}\rt)ie^{i\theta}$ for all $\theta\in\R$. 
Note for later use that $\mu$ is odd if $\widetilde\Phi$ is even.
 In all cases, we have, 
	with the notations of
 Lemma~\ref{l:lambdaPsi} (and Remark~\ref{r:Psi_bdry} in the case where $J\varsubsetneq \mathbb S^1$),
	\begin{align*}
		\Phi(e^{i\theta})
		=\widetilde\Phi(i\hat\Psi(\theta))
		=\widetilde\Phi(e^{i(\frac\pi 2 + \varphi_\Psi(\theta))})\qquad
		\forall \theta\in I,
	\end{align*}
	so
	\begin{align*}
		\partial_\theta\Phi(e^{i\theta})
		&
		=-\varphi_\Psi'(\theta)\mu(i\hat\Psi(\theta))\hat\Psi(\theta)\qquad
		\forall \theta\in I.
	\end{align*}
	Using that
	$\partial_\theta \Gamma_j(e^{i\theta}) =\hat\lambda_j(\theta)\hat\Psi(\theta)$ 
	and $|\hat\lambda|^2=1$, this becomes
	\begin{align*}
		\partial_\theta\Phi(e^{i\theta})
		&
		=\hat\alpha_1(\theta)\partial_\theta\Gamma_1(e^{i\theta})
		+\hat\alpha_2(\theta)\partial_\theta\Gamma_2(e^{i\theta}),\\
		\text{where }
		\hat\alpha_j(\theta)
		&
		=
		-\varphi_\Psi'(\theta)\mu(i\hat\Psi(\theta))\hat\lambda_j(\theta).
	\end{align*}
	We clearly have $\hat\alpha_j\in C^0(I;\R)$, and show next that this function is $2\pi$-periodic 
	in the case $J=\mathbb S^1$,
	distinguishing the cases where $k:=2\deg(\Psi)$ is even or odd.
	
	Note that $\varphi_\Psi(\theta+2\pi)=\varphi_\Psi(\theta)+k\pi$, so $\varphi_\Psi'$ is $2\pi$-periodic
 in both cases. 
	If $k$ is even,
	then $\hat\Psi$ is $2\pi$-periodic.
	Since $\hat\lambda\otimes \hat\Psi$ is $2\pi$-periodic this implies that $\hat\lambda$ is also $2\pi$-periodic, and therefore so is $\hat\alpha_j$.
	If $k$ is odd, then $\hat\Psi(\theta+2\pi)=-\hat\Psi(\theta)$ 
	and again since $\hat\lambda\otimes \hat\Psi$ is $2\pi$-periodic this implies that $\hat\lambda(\theta+2\pi)=-\hat\lambda(\theta)$.
	Moreover in that case we assume that
	$\widetilde\Phi$ is even and therefore $\mu$ is odd, so we also find that $\hat\alpha_j$ is $2\pi$-periodic. 
	
	Hence $\alpha_j(e^{i\theta})=\hat\alpha_j(\theta)$ is well-defined and continuous on $J$, which proves that $\Phi\in \mathrm{ENT}_\Gamma^1$.
\end{proof}

Combining Lemmas~\ref{l:conslawsC1} and \ref{l:entropiesAG}, if $m:\Omega\to J$ solves the family of conservation laws \eqref{eq:conslaws}, then the map $\tilde m=i\Psi(m)$ is an $\mathbb S^1$-valued zero-energy state of the Aviles-Giga energy  \cite[\S~2]{JOP02}, or an $\mathbb{RP}^1$-valued zero-energy state of the unoriented Aviles-Giga energy \cite[\S~1]{GMPS23}. We will use this structure of $\Psi(m)$ in the next two subsections to show Theorems~\ref{t:NE} and \ref{t:NE_bdry}.

\begin{rem}\label{r:geneik}
In \cite[Theorem~4]{LLP22} we used a similar property (link with zero-energy states of Aviles-Giga)
 to prove regularity of solutions to a generalized eikonal equation $N(\nabla u)=1$, where $N$ is a strictly convex $C^1$ norm on $\R^2$: if all entropy productions (associated to that eikonal equation as in \cite[\S~2]{LLP22})
vanish
then $\nabla u$ is continuous outside a locally finite set.
Using an appropriate equivalent of Lemma~\ref{l:entropiesAG}, one could obtain the same result  for  solutions of any equation of the form $\mathcal A(\nabla u)=1$, where $\lbrace \mathcal A=1\rbrace$ is a closed strictly convex $C^2$ curve  in $\R^2$.
Vanishing of all entropy productions can, in turn, be inferred from a zero-energy assumption (as in \cite{LLP22}), 
or from a regularity assumption $\nabla u\in W^{1/3,3}_{\mathrm{loc}}$ 
 and a commutator argument as in \cite{DLI15} (in fact $\nabla u \in B^{1/3}_{3,c_0,\mathrm{loc}}$ is enough).
\end{rem}

\subsection{The case without boundary: proof of Theorem~\ref{t:NE}}

Recall from \eqref{eq:degPsi} and Remark~\ref{r:Psi}, we use the convention that the winding number $\deg(\Psi)$ of a continuous loop 
$\Psi\colon \mathbb S^1\to \mathbb{RP}^1$ is  a half-integer,
and $\Psi$ 
can be identified with a continuous loop $\Psi\colon \mathbb S^1\to\mathbb S^1$ if and only if $\deg(\Psi)\in\mathbb Z$. In this subsection, we prove the following result, from which Theorem~\ref{t:NE} follows immediately.

\begin{prop}\label{p:rigid_m}
Assume that $\Gamma\in C^2(\mathbb S^1;\R^{2\times 2})$ satisfies $|\partial_\theta\Gamma|=1$,
$\det(\partial_\theta\Gamma)=0$ and $\det(\partial_\theta^2\Gamma)\ne 0$ on $\mathbb S^1$.
Let $m\colon \Omega\to\mathbb S^1$ solve the family of conservation laws \eqref{eq:conslaws}.
Then $m$ is locally Lipschitz outside a locally finite set $\mathcal S\subset\Omega$.

Moreover, let $\lambda,\Psi\in C^1(\mathbb S^1;\mathbb{RP}^1)$ satisfy $\partial_\theta\Gamma=\lambda\otimes\Psi$ and $k:=2\deg(\Psi)\in\Z\setminus\{0\}$, and identify $\Psi$ with a map in $C^1(\mathbb S^1;\mathbb S^1)$ if $k$ is even.
The map $m$ satisfies the following additional properties depending on the value of $k$ :
\begin{itemize}
\item[(a)] If $|k|\notin\lbrace 1,2\rbrace$, then $\mathcal S=\emptyset$ and $m$ is constant along characteristic lines directed by $\Psi(m)$.
\item[(b)] If $|k|=2$, then in any convex subset $U\subset\Omega$ containing a singular point $x_0\in U\cap\mathcal S$, there exists $\tau\in\lbrace \pm 1\rbrace$ such that
\begin{align*}
\Psi(m(x))=\tau \frac{x-x_0}{|x-x_0|}\qquad\text{for a.e. } x\in U.
\end{align*}
\item[(c)] If $|k|=1$, then for any disk $B_{2r}(x_0)\subset\Omega$ centered at a singular point $\lbrace x_0\rbrace = B_{2r}(x_0) \cap\mathcal S$, we have $\Psi(m)=\lbrace \pm v\rbrace$ in $B_r(x_0)$, with $v$ as in Theorem~\ref{t:NE} (c).
\end{itemize}
\end{prop}

Note that  in case $(b)$ where $|\deg(\Psi)|=1$, the map $\Psi$ induces a $C^1$ diffeomorphism $\mathbb S^1\to\mathbb S^1$.
And in case $(c)$ where $|\deg(\Psi)|=1/2$, the map $\Psi$ is a $C^1$ diffeomorphism 
$\mathbb S^1\to\mathbb {RP}^1$.
Therefore the proof of Theorem~\ref{t:NE} follows directly from Proposition~\ref{p:rigid_m} and Proposition \ref{p:conslaws}
via the identification $Du=\mathrm{cof}\:\Gamma(m)=\gamma(\theta)$ with $m=e^{i\theta}$.

In the next two subsections we provide the proof of Proposition~\ref{p:rigid_m}, 
first obtaining some regularity for $\Psi(m)$
 which readily implies the cases $(b)$ and $(c)$, and then dealing with the remaining case $(a)$.

%

\subsubsection{Regularity of $\Psi(m)$}

{\color{black}Let $m\colon\Omega\to\mathbb S^1$ be as in Proposition~\ref{p:rigid_m},
in particular it solves the family of conservations laws \eqref{eq:conslaws}.}
Thanks to Lemmas~\ref{l:conslawsC1} and \ref{l:entropiesAG}, the map $\ti m = i\Psi(m)$ solves $\na\cdot\widetilde\Phi(\ti m)=0$ for all $\widetilde\Phi\in C^1(\mathbb S^1;\R^2)$ such that 
$\partial_\theta\widetilde\Phi(e^{i\theta})\cdot e^{i\theta}=0$ (restricting to even maps $\widetilde\Phi$ in the case where $k$ is odd, i.e. $\Psi$ is $\mathbb{RP}^1$-valued).
As a consequence of \cite[Theorem~1.3]{JOP02} or \cite[Theorem~1.2]{GMPS23}, we conclude that $\Psi(m)$ is 
locally Lipschitz continuous outside a locally finite set $\mathcal S\subset\Omega$, and moreover:
\begin{itemize}
	\item if $k$ is even, i.e. $\Psi\in C^1(\mathbb S^1;\mathbb S^1)$,  then $v=\Psi(m)$ is  as in part $(b)$ of Proposition~\ref{p:rigid_m};
	\item if $k$ is odd, i.e. $\Psi\in C^1(\mathbb S^1;\mathbb{RP}^1)$, then $\Psi(m)=\lbrace \pm v\rbrace$ for a map $v\colon\Omega\to\mathbb S^1$ which is as in part $(c)$ of Proposition~\ref{p:rigid_m}.
\end{itemize}
In particular, this concludes the proof of Proposition~\ref{p:rigid_m} in cases $(b)$ and $(c)$, where $\Psi$ is a $C^1$ diffeomorphism, and it remains to treat case $(a)$.

\subsubsection{The case $|k|\notin \lbrace 1,2\rbrace$}
\label{ss:high_deg}

In that case, $\Psi$ is not injective, and the family of conservation laws \eqref{eq:conslaws} contains in fact much more information than the one used when applying the results of \cite{JOP02,GMPS23} to $\tilde m$. This is why we can expect more regularity.

One way of taking advantage of the extra information contained in \eqref{eq:conslaws} is to consider a specific family of nonsmooth entropies, similar to the ones 
used in \cite[Lemma~2.5]{DKMO01} and which are also the main tool in \cite{JOP02}. These nonsmooth entropies are related to kinetic formulations of conservation laws \cite{LPT94,perthame02}, and we will use them precisely via the kinetic equation \eqref{eq:kingen} they provide.

\begin{lem}\label{l:entgen}
For any $\xi\in\mathbb S^1$
 and any open arc $A_\xi \subset \mathbb S^1$ with extremities $a_\xi\neq b_\xi \in\Psi^{-1}(\lbrace \pm\xi\rbrace)$, the map $\Phi^\xi\colon\mathbb S^1\to\R^2$ given by
\begin{align*}
\Phi^\xi ( z)=  \xi\mathbf 1_{z\in A_\xi},
\end{align*}
is a generalized entropy in the sense that there exist $\Phi_k\in \mathrm{ENT}_\Gamma$ such that 
\begin{align*}
\Phi_k(z)\to \Phi^\xi(z)\qquad\text{as }k\to\infty,\quad\forall z\in\mathbb S^1.
\end{align*}
\end{lem}
\begin{rem}
In the case where $\Psi$ is a diffeomorphism, there are only one ($|k|=1$) or two ($|k|=2$) choices for $a_\xi$ and $b_\xi$.
But here  $\Psi$ is not injective, $|\deg(\Psi)|=d\geq 3/2$, so $a_\xi$ and $b_\xi$ can be chosen among $2d\geq 3$ points, and this is where we gain a lot of information.
\end{rem}

\begin{proof}[Proof of Lemma~\ref{l:entgen}] The proof is very close to the proof of \cite[Lemma 15]{LLP22}.
Fix $\theta_a,\theta_b\in\R$ such that
\begin{align*}
a_\xi=e^{i\theta_a},\quad b_\xi=e^{i\theta_b},\quad \theta_a <\theta_b <\theta_a +2\pi,
\end{align*}
and denote
\begin{align*}
\hat \Psi(\theta_a)=\tau_\xi \,\xi,\quad \hat\Psi(\theta_b)=\sigma_\xi\, \xi,\qquad \tau_\xi,\sigma_\xi\in\lbrace \pm 1\rbrace.
\end{align*}
We may choose
 $t_0\in\R$ such that
\begin{align*}
\lbrace \hat\lambda\cdot e^{it_0} =0 \rbrace \cap \lbrace \theta_a,\theta_b\rbrace =\emptyset.
\end{align*} 
We fix a smooth nonnegative kernel $\rho\in C_c^\infty(\R)$ with support $\supp\rho\subset (0,1)$ and unit integral $\int\rho=1$, and let $\rho_{\delta}(\theta)=\frac 1\delta \rho(\frac\theta\delta)$.
Then we define 
 $\alpha_\delta^1, \alpha_\delta^{2}\in C^\infty(\R/2\pi\Z;\R)$ by setting
\begin{align*}
\alpha_\delta^1(\theta)
&=\frac{\cos(t_0)}{e^{it_0}\cdot\hat\lambda(\theta_a)}
\tau_\xi\,\rho_{\delta}\lt(\theta-\theta_a\rt)
-\frac{\cos(t_0)}{e^{it_0}\cdot\hat\lambda(\theta_b)}
\sigma_\xi\, \rho_{\delta}\lt(\theta_b-\theta\rt),
\\
\alpha_\delta^2(\theta)
&=\frac{\sin(t_0)}{e^{it_0}\cdot\hat\lambda(\theta_a)}
\tau_\xi\,\rho_{\delta}\lt(\theta-\theta_a\rt)
-\frac{\sin(t_0)}{e^{it_0}\cdot\hat\lambda(\theta_b)}
\sigma_\xi\, \rho_{\delta}\lt(\theta_b-\theta\rt),
\quad\forall \theta\in (\theta_a,\theta_a +2\pi ],
\end{align*}
and extended as $2\pi$-periodic functions.
Note that these functions are supported in $(\theta_a,\theta_a+\delta)\cup (\theta_b-\delta,\theta_b)+2\pi\Z$.

Then we define 
$\Phi^{\xi}_\delta\colon
(\theta_a,\theta_a +2\pi]\to\R^2$ 
by setting
\begin{align*}
\Phi^{\xi}_\delta(\theta)
=\int_{\theta_a}^{\theta} 
\Big( 
\alpha_\delta^1(t)\partial_{t}\Gamma_1(e^{it})
+ 
\alpha_\delta^2(t)\partial_{t}\Gamma_2(e^{it})
\Big)\,
 dt \qquad\forall\theta \in \lt(\theta_a, \theta_a+2\pi\rt].
\end{align*}
Using the identity \eqref{eq:lambdaPsi} defining $\hat\lambda$ and  $\hat\Psi$, we see that
it satisfies
\begin{align*}
\Phi^{\xi}_\delta(\theta)
&
=\tau_\xi \int \frac{e^{it_0}\cdot\hat\lambda(t)}{e^{it_0}\cdot\hat\lambda(\theta_a)}\hat\Psi(t)\rho_\delta(t-\theta_a)\, dt,
\qquad\text{if }\theta\in [\theta_a+\delta,\theta_b-\delta ],\\
\Phi^{\xi}_\delta(\theta)
&
=\tau_\xi \int \frac{e^{it_0}\cdot\hat\lambda(t)}{e^{it_0}\cdot\hat\lambda(\theta_a)}\hat\Psi(t)\rho_\delta(t-\theta_a)\, dt
\\
&\quad
-
\sigma_\xi \int \frac{e^{it_0}\cdot\hat\lambda(t)}{e^{it_0}\cdot\hat\lambda(\theta_b)}\hat\Psi(t)\rho_\delta(\theta_b-t)\, dt,
\qquad\text{if }\theta\in [\theta_b,\theta_a +2\pi ].
\end{align*}
Since
$\tau_\xi \hat\Psi(\theta_a)
=\sigma_\xi\hat\Psi(\theta_b)=\xi$
by definition of $\tau_\xi,\sigma_\xi$, 
we deduce the limit
\begin{align}\label{eq:limPhideltaxi}
\lim_{\delta\to 0}
\Phi_\delta^{\xi}(\theta)
&=
\begin{cases}
\xi  &\quad\text{ if }\theta\in \lt(\theta_a,\theta_b\rt) ,\\
0
&\quad \text{ if }\theta \in \lt[\theta_b,\theta_a+2\pi\rt].\\
\end{cases}
\end{align}
This corresponds exactly to $\Phi^\xi(e^{i\theta})$.

The function $\Phi_\delta^{\xi}$
may not be $2\pi$-periodic, that is why  we need to modify the functions $\alpha_\delta^j$.
Since, by the above,
\begin{align*}
\int_{\R/2\pi\Z}\left(\sum_{j=1}^2\alpha_\delta^j
\partial_\theta\Gamma_j\right) \, d\theta 
=
\Phi_\delta^\xi(\theta_a+2\pi)\to 0\quad\text{as }\delta\to 0,
\end{align*}
by Lemma~\ref{l:modif} there exist $\tilde\alpha^j_\delta\in C^1(\mathbb S^1;\R)$ such that $(\tilde\alpha^j_\delta-\alpha_\delta^j)$ tends to 0 uniformly as $\delta\to 0$, and
$\sum_{j=1}^2\ti{\alpha}_\delta^j
\partial_\theta\Gamma_j 
$
has zero average on $\mathbb S^1$. Then the map 
\begin{align*}
\widetilde\Phi_\delta^\xi(e^{i\theta})
=\int_{\theta_a}^{\theta} 
\Big( 
\tilde\alpha_\delta^1(e^{it})\partial_{t}\Gamma_1(e^{it})
+ 
\tilde\alpha_\delta^2(e^{it})\partial_{t}\Gamma_2(e^{it})
\Big)\, dt 
\end{align*}
is well-defined,  $\widetilde\Phi_\delta^\xi\in \mathrm{ENT}_\Gamma$, and $\widetilde\Phi_\delta^\xi(e^{i\theta})-\Phi_\delta^\xi(\theta)\to 0$  as $\delta\to 0$, for all $\theta\in (\theta_a,\theta_a+2\pi]$.
Thanks to \eqref{eq:limPhideltaxi}, we deduce that $\widetilde\Phi_\delta^\xi(z)\to\Phi^\xi(z)$ as $\delta\to 0$, for all $z\in\mathbb S^1$.
\end{proof}

Combining Lemma~\ref{l:entgen} and \eqref{eq:conslaws}, we see by dominated convergence that for every $\xi\in\mathbb S^1$ and $ A_\xi $ an open arc with extremities $a_\xi\neq b_\xi\in\Psi^{-1}(\lbrace \pm\xi\rbrace)$, we have
\begin{align}\label{eq:kingen}
\xi\cdot\nabla_x \mathbf 1_{m(x)\in A_\xi} =0 \qquad\text{in }\mathcal D'(\Omega).
\end{align}
We deduce the following:

\begin{lem}\label{l:arcalongsegment}
Let $x_1$ be a Lebesgue point of $m$ and $x_2\ne x_1$ be such that $[x_1,x_2]\subset\Omega$. Let $\xi=\frac{x_2-x_1}{|x_2-x_1|}$. Then, for any open arc $A_\xi \subset \mathbb{S}^1$ with extremities $a_\xi\neq b_\xi\in\Psi^{-1}(\lbrace \pm\xi\rbrace)$, we have
\begin{align*}
	m(x_1)\in A_\xi \quad \Longrightarrow \quad\text{the set } \{x: m(x)\in A_{\xi}\}\text{ has density }1\text{ at }x_2.
\end{align*}
\end{lem}
\begin{proof}[Proof of Lemma~\ref{l:arcalongsegment}]
The proof is exactly the same as the proof of \cite[Proposition~3.1]{JOP02}. We include some details for the convenience of the reader. According to \eqref{eq:kingen}, the function $\chi^\xi$ given by
\begin{align*}
\chi^\xi(x)=\mathbf 1_{m(x)\in A_\xi},
\end{align*}
is constant in the direction of $\xi$ in a neighborhood of the line segment $\lt[x_1,x_2\rt]$, that is,
$\chi^\xi(x)=\tilde\chi(x\cdot i\xi)$ for a.e. $x$ 
in a $\delta$-neighborhood of $\lt[x_1,x_2\rt]$ 
for some $\delta>0$ 
and some measurable function $\tilde \chi \colon (t_1-\delta,t_1+\delta)\to\lbrace 0,1\rbrace$, 
where $t_1=x_1\cdot i\xi=x_2\cdot i\xi$.  
Note that since $A_{\xi}$ is an open set and $x_1$ is a Lebesgue point of $m$, so we have that $t_1=x_1\cdot i \xi$ is a Lebesgue point of $\tilde\chi$ and $\tilde\chi\lt(t_1\rt)=1$. 
It follows that the set $\lt\{x: m(x)\in A_{\xi}\rt\}$ has density $1$ at $x_2$.
\end{proof}

What makes Lemma~\ref{l:arcalongsegment} more powerful in the case where $\Psi$ has a high winding number is that there are several different choices of open arcs $A_\xi$.
We use this flexibility in the following form.

\begin{lem}\label{l:disjoint_arcs}
There is a constant $c\in (0,1)$ depending on the map $\Psi$, with the following property.
For any $z_1 \neq z_2\in\mathbb S^1$, there exist $\xi_1 \neq \xi_2\in\mathbb{RP}^1$ and open arcs $A_{\xi_1},A_{\xi_2}\subset\mathbb S^1$ with extremities in 
$\Psi^{-1}(\xi_1 ),\Psi^{-1}(\xi_2)$, such that:
\begin{align*}
&
z_1\in A_{\xi_1},\; z_2\in A_{\xi_2},
\quad
A_{\xi_1}\cap A_{\xi_2} =\emptyset,
\\
\text{and}
\quad
&
\dist_{\mathbb{RP}^1}(\xi_1,\xi_2)\geq c |z_1-z_2|.
\end{align*}
\end{lem}
\begin{proof}[Proof of Lemma~\ref{l:disjoint_arcs}]
We first set some notations. 
Recall that  $\Psi\colon\mathbb S^1\to\mathbb{RP}^1$ is $C^1$ with uniformly monotone phase and winding number $k/2$ with $|k|\geq 3$.
Let us assume that the phase is increasing and $k\geq 3$, 
the case of a decreasing phase and $k\leq -3$ being completely similar.
For any $\xi =\lbrace \pm e^{i\beta}\rbrace\in\mathbb{RP}^1$ we can write the preimage $\Psi^{-1}(\xi)$ as
\begin{align*}
\Psi^{-1}(\lbrace\pm e^{i\beta}\rbrace)=\lbrace e^{i\alpha_\ell(\beta)}\rbrace_{\ell\in\mathbb Z},
\end{align*}
where the angle functions $\alpha_\ell\in C^1(\R;\R)$, $\ell\in\Z$, 
are uniformly increasing and satisfy
\begin{align*}
\alpha_\ell + c_0 <\alpha_{\ell+1},\qquad \alpha_{\ell +k}=\alpha_\ell +2\pi,\qquad\forall \ell\in\Z,
\end{align*}
for some small enough constant $c_0 >0$ depending on $\Psi$.
The functions $\alpha_\ell$ can simply be chosen as the inverses of the functions $\varphi_\Psi -\ell\pi$, where $\varphi_\Psi\in C^1(\R;\R)$ is a uniformly increasing phase of $\Psi$ as in Lemma~\ref{l:lambdaPsi}.
They satisfy also $\alpha_\ell(\cdot +\pi)=\alpha_{\ell + 1}(\cdot)$.

Without loss of generality we assume that $z_1=e^{i\theta_1}$, $z_2=e^{i\theta_2}$ for some
$\theta_1<\theta_2\leq\theta_1 +\pi$, so that $|z_1-z_2|$ is of the order of $\theta_2-\theta_1$. 
By a continuity argument, we can find $\xi_0=\lbrace\pm e^{i\beta_0}\rbrace\in\mathbb{RP}^1$ and $\ell_1,\ell_2\in\Z$ such that $\ell_1 <\ell_2 <\ell_1+ k$ and
\begin{align*}
\theta_j \in [\alpha_{\ell_j}(\beta_0) + c_1 (\theta_2-\theta_1) ,\alpha_{\ell_j +1}(\beta_0)-c_1(\theta_2-\theta_1)],\quad j=1,2,
\end{align*}
for some small enough constant $c_1 >0$ depending on $\Psi$.
 
Since $k\geq 3$, we must have
either $\alpha_{\ell_1+1}(\beta_0) +c_0 < \alpha_{\ell_2}(\beta_0)$ or $\alpha_{\ell_2 +1}(\beta_0)+c_0<\alpha_{\ell_1 +k}(\beta_0)$.
We assume that we are in the first case,
\begin{align*}
\alpha_{\ell_1+1}(\beta_0) +c_0 < \alpha_{\ell_2}(\beta_0),
\end{align*}
 the other case being completely similar.

For $\xi\in\mathbb{RP}^1$ such that $\dist_{\mathbb{RP}^1}(\xi, \xi_0)<\pi/4$, 
we can write $\xi=\lbrace \pm e^{i\beta}\rbrace$ 
for some $\beta\in (\beta_0-\pi/4,\beta_0+\pi/4)$,
and the angles $\alpha_\ell$ are uniformly increasing functions of $\beta$. 
If $c_1$ is small enough,
  then we can choose 
  $\xi_1=\lbrace \pm e^{i\beta_1}\rbrace$, $\xi_2=\lbrace \pm e^{i\beta_2}\rbrace$  with 
  $\beta_1,\beta_{2}\in (\beta_0-\pi/4,\beta_0+\pi/4)$ and
\begin{align*}
\alpha_{\ell_1}(\beta_1) = \alpha_{\ell_1}(\beta_0)+\frac12 c_1(\theta_2-\theta_1), \quad \alpha_{\ell_2+1}(\beta_2) = \alpha_{\ell_2+1}(\beta_0)-\frac12 c_1(\theta_2-\theta_1).
\end{align*}
It follows that
\begin{align*}
\beta_1-\beta_2 = \lt(\beta_1-\beta_0\rt)+\lt(\beta_0-\beta_2\rt)\geq c(\theta_2-\theta_1),
\end{align*}
for some small enough constant $c$ depending on $\Psi$. Choosing $c_1$ smaller compared to $c_0$ if necessary, we may assume $\alpha_{\ell_1+1}(\beta_1)  < \alpha_{\ell_2}(\beta_2)$, so that
\begin{align*}
\alpha_{\ell_1}(\beta_1)< \theta_1 < \alpha_{\ell_1+1}(\beta_1)  < \alpha_{\ell_2}(\beta_2)<\theta_2 < \alpha_{\ell_2 +1}(\beta_2).
\end{align*}
Thus we have
\begin{align*}
\dist_{\mathbb{RP}^1}(\xi_1, \xi_2) = \beta_1-\beta_2\geq c (\theta_2-\theta_1) 
\gtrsim c |z_1-z_2|,
\end{align*}
and 
\begin{align*}
z_j\in A_{\xi_j}:=\exp(i (\alpha_{\ell_j}(\beta_j),\alpha_{\ell_j +1}(\beta_j)) ),\quad j=1,2,
\end{align*}
where the open arcs $A_{\xi_j}$ have their extremities in $\Psi^{-1}(\xi_j)$.

Moreover, since
$\alpha_{\ell_1+1}(\beta_1)  < \alpha_{\ell_2}(\beta_2)$, and
 by monotonicity of the $\alpha_\ell$ we also have
\begin{align*}
\alpha_{\ell_2+1}(\beta_2)<\alpha_{\ell_2+1}(\beta_0) 
\leq
\alpha_{\ell_1 +k}(\beta_0) =\alpha_{\ell_1}(\beta_0) +2\pi <\alpha_{\ell_1}(\beta_1) +2\pi,
\end{align*}
we deduce $A_{\xi_1}\cap A_{\xi_2}=\emptyset$.
\end{proof}

We can now combine Lemmas~\ref{l:arcalongsegment} and \ref{l:disjoint_arcs} 
to deduce that $m$ is locally Lipschitz.

Let $x_1,x_2\in\Omega$ be two Lebesgue points of $m$, and $z_j=m(x_j)$ for $j=1,2$. 
If $m(x_1)\neq m(x_2)$, then
applying Lemma~\ref{l:disjoint_arcs}, we obtain $\xi_j\in \mathbb{RP}^1$, 
open arcs $A_{\xi_j}\subset\mathbb S^1$ 
with extremities in $\Psi^{-1}(\xi_j)$ 
such that
$z_j\in A_{\xi_j}$
 and
\begin{align}\label{eq:dist_xij_zj}
\dist_{\mathbb{RP}^1}( \xi_1, \xi_2)\geq c |z_1-z_2|.
\end{align}
As $\xi_1\neq \xi_2$ in $\mathbb{RP}^1$, the two lines $x_1+\R \xi_1$, $x_2+\R\xi_2$ intersect in a point $x_0\in\R^2$. 
If the segments $[x_0,x_1]$ and $[x_0,x_2]$ were contained in $\Omega$,
then one would deduce from Lemma~\ref{l:arcalongsegment} that both sets
$\lbrace m\in A_{\xi_j}\rbrace$ have density one at $x_0$,
but this is impossible since $A_{\xi_1}\cap A_{\xi_2}=\emptyset$.
So at least one of
 the segments $[x_0,x_1]$, $[x_0,x_2]$ cannot be contained in $\Omega$, which implies that 
\begin{align*}
\dist(\lbrace x_1,x_2\rbrace,\partial\Omega) \leq \max\{|x_0-x_1|,|x_0-x_2|\}.
\end{align*}
Let $\theta\in(0,\pi/2]$ be the angle between $\xi_1$ and $\xi_2$. 
From elementary trigonometry in the triangle with vertices $\{x_0, x_1, x_2\}$, we deduce that
\begin{align*}
\dist_{\mathbb{RP}^1}(\xi_1,\xi_2)=\theta
\lesssim \sin\left(\theta\right)
&
\leq \frac{|x_1-x_2|}{\max\{|x_0-x_1|,|x_0-x_2|\}}
\\
&
\leq 
\frac{|x_1-x_2|}{\dist(\lbrace x_1,x_2\rbrace,\partial\Omega)}.
\end{align*}
Recalling \eqref{eq:dist_xij_zj}  we deduce that
\begin{align*}
|m(x_1)-m(x_2)|=|z_1-z_2|\lesssim 
\frac{|x_1-x_2|}{c\dist(\lbrace x_1,x_2\rbrace,\partial\Omega)}.
\end{align*}
The above estimate holds automatically if $m(x_1)=m(x_2)$,
 hence $m$ is locally Lipschitz.

In particular, $\Psi(m)$ is locally Lipschitz, and by \cite{JOP02,GMPS23}, it is constant along characteristics directed by $\xi=\Psi(m)$. 
Along these characteristics, the map $m$ is continuous with values into the finite set $\Psi^{-1}(\lbrace\pm\xi\rbrace)$ and must therefore be constant.
This concludes the proof of Proposition~\ref{p:rigid_m} part $(a)$.
\qed

\subsection{The case with boundary: proof of Theorem~\ref{t:NE_bdry}}\label{ss:NE_bdry}

Let $J=\lbrace e^{i\theta}\rbrace_{\theta\in [a,b]}\varsubsetneq \mathbb S^1$, 
so that we have
$Du=\mathrm{cof}\:\Gamma(m)$ \eqref{eq:cofDuGammam} for some $m\colon\Omega\to J$. 
As $\Gamma(e^{i\theta})=\mathrm{cof}\,\gamma(\theta)$, the assumptions of Theorem~\ref{t:NE_bdry} amount to $\Gamma\in C^2(J;\R^{2\times 2})$ satisfying $|\partial_{\theta}\Gamma|=1$, $\det(\partial_{\theta}\Gamma)=0$ and $|\det(\partial_{\theta}^2\Gamma)|>0$ on $J$. 
Moreover, $\Gamma$ satisfies the quartic estimate \eqref{eq:det4}, see Remark~\ref{r:nondegen_NE} and \textsection~\ref{s:det4NE}. 

Thanks to  Lemma~\ref{l:besov} we have
$m\in B^{1/3}_{4,\infty,\mathrm{loc}}(\Omega;J)$.
Applying Proposition~\ref{p:conslaws}, we deduce that 
\begin{align*}
\nabla\cdot \Phi(m)=0\qquad\forall \Phi\in \mathrm{ENT}_{\Gamma}.
\end{align*}
This can be used as in \S~\ref{ss:high_deg} to deduce that $m$ is locally Lipschitz.
In fact the proof is simpler in this case, and we sketch next how to adapt the main steps.

Since we have $|\partial_\theta\Gamma|=1$, $\det(\partial_\theta\Gamma)=0$ 
and $|\det(\partial^2_\theta\Gamma)|>0$ on $J$, as in Lemma~\ref{l:lambdaPsi} 
(see Remark~\ref{r:Psi_bdry}) we can find
$C^1$ maps $\hat\lambda,\hat\Psi \colon [a,b]\to\mathbb S^1$
with uniformly monotone phases and such that 
\begin{align*}
\partial_\theta\Gamma(e^{i\theta})=\hat\lambda(\theta)\otimes\hat\Psi(\theta)\qquad\forall\theta\in [a,b].
\end{align*}
This also defines $\Psi\in C^1(J;\mathbb S^1)$ by $\Psi(e^{i\theta})=\hat\Psi(\theta)$.

In that setting with boundary, 
Lemma~\ref{l:entgen} and the kinetic formulation \eqref{eq:kingen} become valid for 
any $\xi\in\Psi(J)$ and any arc 
$\exp(i\mathcal I)=A_\xi\subset J$, where $\mathcal I$ has one of the forms $(\theta_1, \theta_2)$, $[a, \theta_1)$ or $(\theta_2, b]$ for $e^{i\theta_j}\in\Psi^{-1}(\lbrace\pm\xi\rbrace)$.
(The proof is actually easier since we don't need to make the approximating entropies periodic.)
Writing $m(x)=e^{i\theta(x)}$ for $\theta\colon\Omega\to [a,b]$, this implies in particular
\begin{align*}
\hat\Psi(\alpha)\cdot \nabla_x \mathbf 1_{\theta(x)<\alpha} =0\qquad\forall \alpha\in [a,b],
\end{align*}
and the corresponding version of Lemma~\ref{l:arcalongsegment}. 
Namely, if $x_1$ is a Lebesgue point of $\theta$ such that $\theta(x_1)<\alpha$,
then the set $\lbrace\theta <\alpha\rbrace$ has density one at
all points $x_2\in x_1 +\R\hat\Psi(\alpha)$ such that $[x_1,x_2]\subset\Omega$.
The same holds for the set $\lbrace \theta >\alpha\rbrace$ if $\theta(x_1)>\alpha$.
Thanks to the uniform monotonicity of $\hat\Psi$'s phase, 
we obtain the following simpler version of Lemma~\ref{l:disjoint_arcs}.
For any $\theta_1 < \theta_2\in [a,b]$ we can find $\theta_1 <\alpha_1 <\alpha_2 <\theta_2$ such that 
\begin{align*}
\dist_{\mathbb{RP}^1}(
\lbrace\pm\hat\Psi(\alpha_1)\rbrace,
\lbrace\pm\hat\Psi(\alpha_2)\rbrace
) \geq c (\theta_2-\theta_1).
\end{align*}
As in \S~\ref{ss:high_deg}, we apply this to the values $\theta_j=\theta(x_j)$ at
Lebesgue points $x_1,x_2$ of $\theta$, 
 deduce that the lines $x_j +\R \hat\Psi(\alpha_j)$ must intersect $\partial\Omega$ before crossing,
and conclude that $\theta$ is locally Lipschitz, 
and so are $m$ and $\Psi(m)$.

Finally, Lemmas~\ref{l:conslawsC1} and \ref{l:entropiesAG} ensure that $\ti m=i\Psi(m)$ is an $\mathbb S^1$-valued zero-energy state of the Aviles-Giga energy in the sense of \cite[\S~2]{JOP02}. Arguing exactly as the end of \S~\ref{ss:high_deg}, we deduce that $m$ is constant along characteristics directed by $\Psi(m)$, and this concludes the proof of Theorem~\ref{t:NE_bdry} for $Du=\mathrm{cof}\:\Gamma(m)$.
\qed

\section{Examples of nowhere elliptic curves without rank-one connections}
\label{s:ex_NE}

In this section we give several examples of curves satisfying the assumptions of Theorem~\ref{t:NE}, and prove Proposition~\ref{p:NE} which states that,
among nondegenerate nowhere elliptic curves, 
the subset of curves without rank-one connections is open.

\subsection{The curves $\gamma_{k}$}
\label{ss:gammak}

\begin{lem}\label{l:gamma_k}
The curves parametrized by
\begin{equation*}
\gamma_{k}(t)
=\frac{1}{2} [e^{it}]_c +\frac{1}{2(k+1)}[e^{(k+1)it}]_a ,
\end{equation*}
for any integer $k\geq 1$,
have no rank-one connections, 
are nowhere elliptic, and satisfy $\det(\gamma''_k)<0$ on $\R/2\pi\Z$. 
That is, they satisfy the assumptions of Theorem~\ref{t:NE}.
\end{lem}

\begin{proof}[Proof of Lemma~\ref{l:gamma_k}]
Since 
\begin{align*}
\gamma_{k}'(t)&
=\frac{1}{2}\lt(\lt[i e^{it}\rt]_c+\lt[i e^{(k+1)it}\rt]_a\rt),\\
\gamma_{k}''(t)
&=\frac{1}{2}\lt(\lt[-e^{it}\rt]_c+\lt[-(k+1)e^{(k+1)it}\rt]_a\rt),
\end{align*}
it is straightforward to check that 
$\det\lt(\gamma_{k}'\rt)=0$ and 
$\det\lt(\gamma_{k}''\rt)=\frac 14\lt(1-(k+1)^2\rt)<0$ 
for all $k\geq 1$. 

It remains to check that the curve $\gamma_{k}(\R/2\pi\mathbb{Z})$ does not have rank-one connections. Note that, for all $t, h\in\R$, we have
\begin{align*}
\det\lt(\gamma_{k}(t+h)-\gamma_{k}(t)\rt)
&=
\frac{1}{4}\lt|e^{i\lt(t+h\rt)}-e^{it}\rt|^2 -\frac{1}{4(k+1)^2}\lt|e^{(k+1)i\lt(t+h\rt)}-e^{(k+1)it}\rt|^2
\nn\\
&=\frac{1}{4}\lt(\lt|e^{ih}-1\rt|^2-(k+1)^{-2}\lt|e^{(k+1)ih}-1\rt|^2\rt)\nn\\
&=\sin^2\lt(\frac{h}{2}\rt) - (k+1)^{-2}\sin^2\lt(\frac{(k+1) h}{2}\rt)
=:f(h).
\end{align*}
We show next that $f(h)>0$ for all $h\in (0,2\pi)$, which implies that the curve $\gamma_{k}(\R/2\pi\mathbb{Z})$ has no rank-one connections.

First note that $f(2\pi-h)=f(h)$, so it suffices to show that $f>0$ on $(0,\pi]$.
Then we argue separately on $(\pi/2,\pi]$ and $(0,\pi/2]$.

For $h\in (\pi/2,\pi]$ we have, using that $\sin(x)\geq \frac{2 x}{\pi}$ for all $x\in (0,\pi/2]$,
\begin{align*}
f(h)
&
\geq 
\frac {h^2}{\pi^2} -(k+1)^{-2}\sin^2\lt(\frac{(k+1) h}{2}\rt)
> \frac 14 -\frac{1}{(k+1)^2}\geq 0.
\end{align*}
On $(0,\pi/2]$ we show that $f'>0$, 
which completes the proof since $f(0)=0$.
The derivative is given by
 \begin{align*}
f'(h)&=\frac{1}{2}\lt( \sin\lt(h\rt)-  \frac{\sin\lt((k+1) h\rt)}{k+1} \rt). 
\end{align*}
Using  that $\sin(x)>\frac{2 x}{\pi}$ for all $x\in \lt(0,\frac{\pi}{2}\rt)$, we know 
$\sin\lt(\frac{\pi}{2(k+1)}\rt)>\frac{1}{k+1}$ and it follows that 
\begin{align*}
\sin\lt(h\rt)
>
\sin\lt(\frac{\pi}{2(k+1)}\rt)
>\frac{\sin\lt((k+1) h \rt)}{k+1}\qquad \forall h\in \lt(\frac{\pi}{2(k+1)}, \frac \pi 2\rt].
\end{align*}
Note that $\cos\lt(x\rt)> \cos\lt( (k+1) x \rt)$ for $x\in \lt(0,\frac{\pi}{2(k+1)}\rt]$ since the cosine function is decreasing on $\lt(0,\frac{\pi}{2}\rt]$. By integrating this inequality we have that 
\begin{align*}
\sin\lt(h\rt)>\frac{\sin\lt((k+1)h \rt)}{k+1}\qquad \forall h\in \lt(0,\frac{\pi}{2(k+1)}\rt].
\end{align*}
Putting the above two estimates together gives $f'(h)>0$ for all $h\in(0,\pi/2]$.
\end{proof}

\subsection{Further examples}\label{ss:ex}

Here we show how to construct many other examples of curves satisfying the assumptions of Theorem~\ref{t:NE}.

\begin{lem}\label{l:furtherexamples} 
	Let $\gamma\in C^2\lt(\mathbb{R}/2\pi \mathbb{Z}; \mathbb{C}\rt)$ be an arc-length parametrization of a simple closed curve, and let 
 $\ti{\gamma}\in C^2\lt(\mathbb{R}/2\pi \mathbb{Z}; \mathbb{C}\rt)$ be such that $\lt|\ti{\gamma}'\rt|\equiv 1$ and $|\tilde\gamma''|>0$ on $\mathbb{R}/2\pi \mathbb{Z}$. 
 Then there exists $k_0=k_0\lt(\gamma, \ti{\gamma}\rt)\in \mathbb{N}$ such that for every integer $k\geq k_0$, the curve in $\R^{2\times 2}$
parametrized by 
\begin{align*}
\alpha_k(t):=\lt[\gamma(t)\rt]_c+k^{-1}\lt[\ti{\gamma}(k t)\rt]_a
\end{align*}
is nowhere elliptic, satisfies $\det(\alpha_k'')<0$ and has no rank-one connections.
\end{lem}
\begin{proof}[Proof of Lemma~\ref{l:furtherexamples}]
To simplify notation let $\hat{\gamma}_{k}(t)=k^{-1}\ti{\gamma}(k t)$. First note that $\lt|\hat{\gamma}_{k}'\rt|=\lt|\gamma'\rt|=1$, so $\det(\alpha_k')=|\gamma'|^2-|\hat\gamma_{k}'|^2=0$, that is, the curve is nowhere elliptic.

Moreover, we have
\begin{align*}
\det(\alpha_k'')&
=|\gamma''|^2 -k^2 |\tilde\gamma''|^2 \leq \sup |\gamma''|^2 -k^2 \inf |\tilde\gamma''|^2
<0,
\end{align*}
provided $k\geq k_0$ for $k_0$ large enough.
Thanks to the identity \eqref{eq:limdetover4} in the proof of Lemma~\ref{l:detcontrols4}
and by compactness of $\R/2\pi\Z$, this implies the existence of
$\delta_1=\delta_1(\gamma, \ti{\gamma})>0$ such that 
\begin{align}
\label{eqnx103}
\det\lt(\alpha_k(s)- \alpha_k(t)\rt)>0\qquad\text{for any }s,t\in \mathbb{R}/2\pi \mathbb{Z}
\text{ s.t. }|e^{is}-e^{it}|\leq \delta_1.
\end{align}
Also note that since $\gamma$ parametrizes a simple closed curve, there exists another constant $\beta_0=\beta_0(\gamma)>0$ such that 
\begin{align}
\label{eqnx101}
\lt|\gamma(s)- \gamma(t) \rt|\geq \beta_0\qquad\text{for any }s,t\in \mathbb{R}/2\pi \mathbb{Z}
\text{ s.t. }|e^{is}-e^{it}|\geq \delta_1.
\end{align}
For $s,t$ as in \eqref{eqnx101} we infer
\begin{align*}
\det(\alpha_k(s)-\alpha_k(t))
&
=
\lt|\gamma(s)- \gamma(t) \rt|^2
- \frac{1}{k^2}|\tilde\gamma(ks)-\tilde\gamma(kt)|^2
\\
&
\geq \beta_0^2-\frac{4}{k^2}\sup |\tilde\gamma|^2,
\end{align*}
and deduce that
\begin{align}
\label{eqnx102}
\det(\alpha_k(s)-\alpha_k(t))
>0\qquad\text{for any }
s,t\in \mathbb{R}/2\pi \mathbb{Z}
\text{ s.t. }|e^{is}-e^{it}|\geq \delta_1,
\end{align}
provided $k\geq k_0$ for $k_0$ large enough.
Putting \eqref{eqnx102} and \eqref{eqnx103} together shows that
the curve parametrized by $\alpha_k$ 
 has no rank-one connections. 
\end{proof}

\subsection{Proof of Proposition~\ref{p:NE} }\label{ss:NE*}

Recall the definitions of the set
\begin{align*}
\mathrm{NE}_*=\left\lbrace \gamma\in C^2(\R/2\pi\Z;\R^{2\times 2}) \colon \det(\gamma')= 0
\text{ and }|\det(\gamma'')|>0
\right\rbrace,
\end{align*}
and its subset
\begin{align*}
\mathrm{NE}_{**}=
\Big\lbrace \gamma\in \mathrm{NE}_*\colon
&
\gamma(\R/2\pi\Z)\text{ has no rank-one connections}
\Big\rbrace.
\end{align*}
In this section we prove Proposition~\ref{p:NE}, which asserts that $\mathrm{NE}_{**}$ is open in $\mathrm{NE}_*$ for the $C^2$ topology.

Let $\gamma,\bar\gamma\in \mathrm{NE}_*$.
It is clear from the expansion \eqref{eq:incr_det_gamma}, the explicit form of $\tilde\e(t,h)$ and the fact that $\det(\gamma')=\det(\bar\gamma')=0$, that for all $t,h$ in $\R$ we have
\begin{align*}
&\left|\det\left(\gamma(t+h)-\gamma(t)\right)-\det\left(\bar\gamma(t+h)-\bar\gamma(t)\right)\right|
\\
&
\leq C (\|\gamma''\|_{\infty}+\|\bar\gamma''\|_{\infty}) \|\gamma''-\bar\gamma''\|_{\infty} h^4,
\end{align*}
for some absolute constant $C>0$.

Now fix $\bar\gamma\in\mathrm{NE}_{**}$  and assume without loss of generality that $\det(\bar\gamma(t)-\bar\gamma(s))>0$ for all $s\ne t\in\R/2\pi\mathbb{Z}$, 
then the last inequality implies that, if $\|\gamma''-\bar\gamma''\|_{\infty}\leq \delta$ for  $\delta>0$ small enough,
then for all $s,t\in\R$ we have
\begin{align*}
	\det(\gamma(t)-\gamma(s))
	\geq \det(\bar\gamma(t)-\bar\gamma(s)) -C(\bar\gamma)\delta |e^{it}-e^{is}|^4.
\end{align*}
By definition of $\mathrm{NE}_{**}$ and Lemma~\ref{l:detcontrols4}
there exists $\bar\kappa>0$ such that
\begin{align*}
	\det(\bar\gamma(t)-\bar\gamma(s))
	\geq \bar\kappa |e^{it}-e^{is}|^4,
\end{align*}
so if $\|\gamma''-\bar\gamma''\|_{\infty}\leq \delta$ for  $\delta>0$  small enough we deduce
\begin{align*}
	\det(\gamma(t)-\gamma(s))
	\geq \frac{\bar\kappa}{4} |e^{it}-e^{is}|^4.
\end{align*}
That is, $\gamma\in \mathrm{NE}_{**}$, and this shows that $\mathrm{NE}_{**}$ is open in $\mathrm{NE}_*$ with respect to the $C^2$ topology.
\qed

%
%

\begin{appendices}

%
%
%
%

\section{Rigidity estimate for elliptic arcs}
\label{a:rig}

We provide here the
extension of \cite[Theorem~1.1]{LLP23} 
needed in the proof of Proposition~\ref{p:reg_ellipt},
namely, a rigidity estimate 
 for $C^2$ elliptic arcs $\mathcal J\subset\R^{2\times 2}$.
In \cite{LLP23} this was established for smooth closed elliptic curves.

%
%

\begin{prop} \label{p:rig_est_J}
	Let $a<b$ and $\gamma\in C^2\lt(\lt[a,b\rt];\R^{2\times 2}\rt)$ be
	injective such that
	$\mathcal J=\gamma([a,b])$ is elliptic:
	\begin{align}\label{eq:elliptJ}
		\det(A-B)\geq c_0|A-B|^2\qquad\forall A,B\in \mathcal J,
	\end{align}
	for some $c_0>0$.
	Then $\mathcal J$ satisfies a rigidity estimate:
	there exists $C=C(\mathcal J)>0$ such that
	\begin{align*}
		\inf_{A\in \mathcal J}\int_{B_{1/2}}| Du -A |^2\,dx  \leq C \int_{B_1} \dist^2(Du, \mathcal J)\, dx,
	\end{align*}
	for all $u\in W^{1,2}(B_1;\R^2)$.
\end{prop}
\begin{proof}[Proof of Proposition~\ref{p:rig_est_J}] We will show in Lemma \ref{RLE1} that we can find a closed $C^2$ curve $\Gamma \subset \R^{2\times 2}$
	such that $\mathcal J\subset\Gamma$ and $\Gamma$ is elliptic:
	\begin{align}\label{eq:elliptGamma}
		\det(A-B)\geq c_1 |A-B|^2\qquad\forall A,B\in \Gamma.
	\end{align}
	Granted this, we can directly apply \cite[Theorem~1.1]{LLP23} which provides $C=C(\Gamma)>0$ such that
	\begin{align}\label{eq:rig_est_Gamma}
		\inf_{A\in \Gamma}\int_{B_{1/2}}| Du -A |^2\,dx  \leq C \int_{B_1} \dist^2(Du, \Gamma)\, dx
		\leq C \int_{B_1}\dist^2(Du,\mathcal J)\,dx,
	\end{align}
	for all $u\in W^{1,2}(B_1;\R^2)$.
	In \cite[Theorem~1.1]{LLP23} this is stated for a smooth curve $\Gamma$, but the proof only requires $C^2$ regularity
	(and could probably be modified to require only $C^1$ regularity).
	Moreover, given $u\in W^{1,2}(B_1;\R^2)$ and $A\in\Gamma$ attaining the infimum in the left-hand side of \eqref{eq:rig_est_Gamma},
	we can integrate on $B_{1/2}$ the elementary inequality
	\begin{align*}
		\dist^2(A,\mathcal J)\leq 2 |Du-A|^2 + 2 \dist^2(Du,\mathcal J),
	\end{align*}
	and use \eqref{eq:rig_est_Gamma} to deduce
	\begin{align*}
		|B_{1/2}|\dist^2(A,\mathcal J) \leq (2C+2)\int_{B_1} \dist^2(Du,\mathcal J)\, dx.
	\end{align*} 
	Taking $\ti A\in \mathcal J$ such that $|A-\ti A|=\dist(A,\mathcal J)$ and combining the above estimate with \eqref{eq:rig_est_Gamma}, we obtain
	\begin{align*}
		\int_{B_{1/2}}| Du -\ti A |^2\,dx
		\leq (6C+4) \int_{B_1}\dist^2(Du,\mathcal J)\,dx,
	\end{align*}
	thus proving Proposition~\ref{p:rig_est_J}.
\end{proof}

\begin{lem}
\label{RLE1}
Let $a<b$ and $\gamma\in C^2\lt(\lt[a,b\rt];\R^{2\times 2}\rt)$ be
injective such that $\mathcal J=\gamma([a,b])$ satisfies \eqref{eq:elliptJ}. Then there exists a closed $C^2$ curve $\Gamma \subset \R^{2\times 2}$
such that $\mathcal J\subset\Gamma$ and $\Gamma$ satisfies \eqref{eq:elliptGamma} for some positive constant $c_1<c_0$.
\end{lem}
\begin{proof} 
First we extend the $C^2$ curve $\gamma$ to $[a-\e,b+\e]$ for some $\e>0$, while conserving the ellipticity \eqref{eq:elliptJ} for $\mathcal J=\gamma([a-\e,b+\e])$.
To that end we set
$\gamma(a-t)=\gamma(a)-t\gamma'(a)+t^2\gamma''(a)/2$ for $0<t\leq\e$.
That way \eqref{eq:elliptJ} is satisfied with a possibly smaller constant $c_0$ (not renamed) for $A,B\in \gamma([a-\e,a+\e])$ 
provided $\e$ is small enough, 
because $\det(\gamma'(a+h)) = \det(\gamma'(a)) +o(1)\neq 0$ for $h\to 0$. 
And \eqref{eq:elliptJ} is satisfied also for $A\in\gamma([a-\e,a])$ and $B\in\gamma([a+\e,b])$ because $A\approx \gamma(a)$ and $\det(\gamma(a)-B)$ is bounded from below by a positive constant. 
The same arguments apply for the extension to $[b,b+\e]$. Note that $\gamma$ is smooth in $[a-\ep,a)\cup(b,b+\ep]$.

Then we use the classical fact  \cite{zhang97,FS08}
 that the ellipticity \eqref{eq:elliptJ} of $\mathcal J$ implies that the conformal-anticonformal decomposition \eqref{eq:gamma_ac} of $\gamma$ is given by
\begin{align*}
\gamma=[\gamma_c]_c + [H\circ\gamma_c]_a,
\end{align*}
where $\gamma_c\in C^2([a-\e,b+\e];\mathbb C)$ is injective with $|\gamma'_c|>0$ on $[a-\e,b+\e]$, and 
\begin{align*}
H\colon \gamma_c([a-\e,b+\e])\to \C,
\end{align*}
 is $k$-Lipschitz for $k=\sqrt{(1-2c_0)/(1+2c_0)}\in (0,1)$.  Next we extend $H$ to a $k$-Lipschitz map over $\mathbb{C}$ with a possibly larger $k<1$ such that the extension is smooth outside a sufficiently small neighborhood of $\gamma_c([a,b])$. The proof is very similar to the proof of \cite[Lemma 3.1]{LLP23}, and we sketch below the key steps. 
 
In the first step, possibly after a reparametrization, we may assume without loss of generality that $\gamma_c(t)$ is an arc-length parametrization of $\gamma_c([a-\ep,b+\ep])$. 
For fixed $\ep$ and small enough $\delta>0$, denote $R_{\delta}:=[a,b]\times(-\delta,\delta)$ and $R_{\delta}^\ep:=\lt([a-\ep,a)\cup(b,b+\ep]\rt)\times(-\delta,\delta)$. The map
\begin{align*}
	\varphi\colon (t,r) &\mapsto \gamma_c(t) +r i\gamma_c'(t),
\end{align*}
is a $C^1$ diffeomorphism between $R_{\delta}\cup R_{\delta}^\ep$ and $\varphi(R_{\delta}\cup R_{\delta}^\ep)$. 
Since $\gamma$ is smooth in $[a-\ep,a)\cup(b,b+\ep]$, this map $\varphi$ is further a smooth diffeomorphism between $R_{\delta}^\ep$ and $\varphi(R_{\delta}^\ep)$. 
For $z\in\varphi(R_{\delta}\cup R_{\delta}^\ep)$, define $\widetilde H(z) = H(\gamma_c(t))$ where $z=\varphi(t,r)$. 
This $\widetilde H$ agrees with $H$ on $\gamma_c([a-\ep, b+\ep])$, and is $C^1$ in $\varphi(R_{\delta}\cup R_{\delta}^\ep)$ and smooth in $\varphi(R_{\delta}^\ep)$ by the regularity of $\gamma_a=H\circ\gamma_c$ and $\varphi^{-1}$. 
Further, we have $\|D(\varphi^{-1})\|\leq 1+C\delta$ for some constant $C$ depending on $\gamma_c$, and it follows that $\widetilde H$ is $\ti k$-Lipschitz with $\ti k=(1+C\delta)k<1$ for small enough $\delta$.  

In the second step, we extend $\widetilde H$ to $\mathbb{C}$ such that the extension is smooth outside a sufficiently small neighborhood of $\gamma_c([a,b])$. By Kirszbraun's theorem, we can first extend  $\widetilde H$ to a $\ti k$-Lipschitz map (not renamed) over $\mathbb{C}$. For $\alpha>0$, we define  $H_{\alpha}(z)=\int_\C \widetilde H(z+\alpha  \chi(z) y)\rho(y)\, dy$ for a smooth kernel $\rho\geq 0$ with support in $B_1$ 
 and $\int\rho(y) \, dy=1$, and some smooth cut-off function $\chi$ with
 $\mathbf 1_{\mathcal U_{\delta^2/4}}\leq 1- \chi\leq \mathbf 1_{\mathcal U_{\delta^2/2}}$ and $\|\na\chi\|_{\infty}\leq 8/\delta^2$, where $\mathcal U_{\delta}:=\{z\in\mathbb{C}: \dist(z, \gamma_c([a,b]))<\delta\}$. Note that $H_{\alpha}$ agrees with $\widetilde H$ in $\mathcal U_{\delta^2/4}$ and is smooth in $\lt(\mathbb{C}\setminus\overline{\mathcal U_{\delta^2/2}}\rt)\bigcup \varphi(R_{\delta/2}^{\ep/2}\setminus R_{\delta/2}^{\delta^2/8})$ for $\alpha$ small enough. In particular, $H_{\alpha}$ agrees with $H$ on $\gamma_c([a-\frac{3\delta^2}{16}, b+\frac{3\delta^2}{16}])$. Further we have
 \begin{align*}
 \lt|H_{\alpha}(z)-H_{\alpha}(z')\rt|&\leq \int_{B_1}|\widetilde H(z+\alpha\chi(z)y)-\widetilde H(z'+\alpha\chi(z')y)|\,\rho(y)\,dy\\
 &\leq \ti k\lt(1+\frac{8\alpha}{\delta^2}\rt)\int_{B_1}|z-z'|\,\rho(y)\,dy= \ti k\lt(1+\frac{8\alpha}{\delta^2}\rt)|z-z'|,
 \end{align*}
so $H_{\alpha}$ is $k_{\alpha}$-Lipschitz for $k_{\alpha}=\ti k\lt(1+\frac{8\alpha}{\delta^2}\rt)<1$ for $\alpha$ sufficiently small.

Finally we use the above $H_{\alpha}$ to construct the closed curve $\Gamma$ containing $\mathcal J$. To that end, we claim that $\gamma_c([a-\e/3,b+\e/3])$ can be extended to a closed $C^2$ curve in $\C$ parametrized by $\tilde\gamma_c \in C^2(\R/L\Z;\C)$ for some $L>b-a+2\e/3$. 
This can be obtained as a consequence of Jordan's theorem stating that the complement of a Jordan arc in the plane is connected, 
but we sketch here a simpler proof in our $C^2$ context. Recall that $\varphi$ is a $C^1$ diffeomorphism between $R_{\delta}\cup R_{\delta}^\ep$ and $\varphi(R_{\delta}\cup R_{\delta}^\ep)$. 
We can find an open neighborhood $\mathcal V$ of $R_{\delta/2}\cup R_{\delta/2}^{2\ep/3}$ with $\mathcal V\subset\subset R_{\delta}\cup R_{\delta}^\ep$, and extend the straight segment $[a-\ep/3, b+\ep/3]$ 
into $\mathcal V\setminus \lt(R_{\delta^2}\cup R_{\delta^2}^{\ep/3}\rt)$ to form a smooth closed curve and apply $\varphi$ to obtain a $C^1$ closed curve inside $\varphi(\mathcal V)$ extending $\gamma_c([a-\ep/3, b+\ep/3])$. 
Finally smoothen that curve outside $[a-\ep/3, b+\ep/3]$ 
(using the same method we used to smoothen $\wt{H}$ via a cut-off function of the form of $\chi$) to turn it into a $C^2$ closed curve $\ti \gamma_c$ extending $\gamma_c([a-\ep/3, b+\ep/3])$. 

We are now equipped with $\tilde\gamma_c \in C^2(\R/L\Z;\C)$ injective such that $\tilde\gamma_c=\gamma_c$ on $[a-\ep/3,b+\ep/3]$.
Then we define $\tilde\gamma_a=H_{\alpha}\circ\ti \gamma_c$. Since $H_{\alpha}$ is smooth in  $\lt(\mathbb{C}\setminus\overline{\mathcal U_{\delta^2/2}}\rt)\bigcup \varphi(R_{\delta/2}^{\ep/2}\setminus R_{\delta/2}^{\delta^2/8})$, we know $\ti \gamma_a$ is $C^2$ outside the interval $[a-\frac{\delta^2}{8}, b+\frac{\delta^2}{8}]$. Further, since $H_{\alpha}$ agrees with $H$ on $\gamma_c([a-\frac{3\delta^2}{16}, b+\frac{3\delta^2}{16}])$,  it follows that $\ti \gamma_a=H_{\alpha}\circ\ti \gamma_c = H\circ\gamma_c=\gamma_a\in C^2([a-\frac{3\delta^2}{16}, b+\frac{3\delta^2}{16}])$. Thus, we deduce that $\ti\gamma_a\in C^2(\R/L\Z;\C)$. Setting
\begin{align*}
	\tilde\gamma =[\tilde\gamma_c]_c +[\tilde\gamma_a]_{a},
\end{align*}
it follows that $\Gamma=\tilde\gamma(\R/L\Z)$ is a closed $C^2$ elliptic curve containing $\mathcal J$.
The ellipticity follows  from the fact that $H_{\alpha}$ is $k_{\alpha}$-Lipschitz with $k_{\alpha}\in (0,1)$.
\end{proof}

\end{appendices}

\bibliographystyle{alphaa}
\bibliography{nonelliptic}

\end{document}